\documentclass[11pt, letterpaper]{amsart}
\usepackage{graphicx, amssymb, color}
\usepackage{amsmath}
\usepackage{graphicx}
\usepackage{enumerate}

\newtheorem{thm}{Theorem}[section]
\newtheorem{cor}[thm]{Corollary}
\newtheorem{lem}[thm]{Lemma}
\newtheorem{prop}[thm]{Proposition}
\theoremstyle{definition}

\newtheorem{ass}[thm]{Assumption}

\theoremstyle{remark}
\newtheorem{rem}[thm]{Remark}
\newtheorem{exa}[thm]{Example}

\numberwithin{equation}{section}
\newcommand{\norm}[1]{\|#1\|}

\newcommand{\Real}{\mathbb R}
\newcommand{\Natural}{\mathbb N}

\newcommand{\F}{\mathcal{F}}

\newcommand{\cJ}{\mathcal{J}}
\newcommand{\prob}{\mathbb{P}}

\newcommand{\expec}{\mathbb{E}}

\newcommand{\bL}{\mathbb{L}}
\newcommand{\cL}{\mathcal{L}}

\newcommand{\fF}{\mathfrak{F}}
\newcommand{\barA}{\overline{A}}

\newcommand{\uk}{\underline{\kappa}}
\newcommand{\ok}{\overline{\kappa}}
\newcommand{\uc}{\underline{c}}
\newcommand{\oc}{\overline{c}}
\newcommand{\sd}{\mathbb{S}_{++}^d}
\newcommand{\sdall}{\mathbb{S}^d}
\newcommand{\tr}{\mbox{Tr}}
\newcommand{\td}{\tilde{d}}

\newcommand{\indic}{\mathbb{I}}
\newcommand{\pare}[1]{\left(#1\right)}
\newcommand{\bra}[1]{\left[#1\right]}
\newcommand{\cbra}[1]{\left\{#1\right\}}

\newcommand{\wh}[1]{\widehat{#1}}

\newcommand{\qprob}{\mathbb{Q}}

\newcommand{\nada}[1]{}
\newcommand{\espalt}[3]{\expec^{#1}_{#2}\bra{#3}}
\newcommand{\hespalt}[3]{\hat{\expec}^{#1}_{#2}\bra{#3}}

\newcommand{\dfn}{\, := \,}

\title[Large time behavior of solutions to semi-linear equations]{Large time
  behavior of solutions to semi-linear equations with quadratic growth in the gradient}
\thanks{The authors are grateful to Gechun Liang, the anonymous Associate Editor and two referees for their valuable comments which help us improving this paper. S. Robertson is supported in part by the National Science Foundation
  under grant number DMS-1312419.}

\author[]{Scott Robertson}
\address[Scott Robertson]{Department of Mathematical Sciences,
Wean Hall 6113,
Carnegie Mellon University,
Pittsburgh, PA 15213,
USA}
\email{scottrob@andrew.cmu.edu}

\author[]{Hao Xing}
\address[Hao Xing]{Department of Statistics,
London School of Economics and Political Science,
10 Houghton st,
London, WC2A 2AE,
UK}
\email{h.xing@lse.ac.uk}

\begin{document}

\begin{abstract}
 This paper studies the large time behavior of solutions to semi-linear Cauchy
 problems with quadratic nonlinearity in gradients. The Cauchy problem
 considered has a general state space and may degenerate on the boundary of
 the state space. Two types of large time behavior are obtained: i)  pointwise
 convergence of the solution and its gradient; ii) convergence of solutions to
 associated backward stochastic differential equations. When the state space
 is $\Real^d$ or the space of positive definite matrices, both types of
 convergence are obtained under growth conditions on  coefficients. These
 large time convergence results have direct applications in risk sensitive control and long term portfolio choice problems.
\end{abstract}


\maketitle

\section{Introduction}\label{sec: intro}
Given an open domain $E\subseteq \Real^d$ and functions $A_{ij}$, $\barA_{ij}$, $B_i$, $V$, $i,j=1, \cdots, d$, from $E$ to $\Real$, define the differential operator
\begin{equation}\label{eq: F E}
 \mathfrak{F} \dfn \frac12 \sum_{i,j=1}^d A_{ij} D_{ij} + \frac12
 \sum_{i,j=1}^d \barA_{ij} D_i  D_j + \sum_{i=1}^d B_i D_i  + V,
\end{equation}
where $D_i = \partial_{x^i}$ and $D_{ij}= \partial^2_{x^ix^j}$. We consider the following Cauchy problem:
\begin{equation}\label{eq: pde E}
 \partial_t v = \mathfrak{F}[v], \quad (t,x) \in (0,\infty) \times E, \quad
 v(0,x) = v_0(x).
\end{equation}
Precise conditions on $E$, the coefficients, and the initial condition $v_0$ will be presented later. In particular, these conditions allow for general domains $E$ and for $A=(A_{ij})_{1\leq i,j\leq d}$ to be both unbounded and degenerate on the boundary of $E$. Our goal is to study the large time asymptotic behavior of solutions $v(t,\cdot)$ to \eqref{eq: pde E}.

 The asymptotic behavior of $v(t, \cdot)$ is closely related to the following ergodic analogue of \eqref{eq: pde E}:
\begin{equation}\label{eq: e-eqn E}
\lambda = \mathfrak{F}[v], \quad x\in E,
\end{equation}
whose solution is a pair $(v, \lambda)$ with $\lambda \in \Real$. In our main result, we prove the existence of $(\hat{v}, \hat{\lambda})$ solving \eqref{eq: e-eqn E} such that $h(t,x):= v(t,x) - \hat{\lambda} t - \hat{v}(x)$, $x\in E$, satisfies
\begin{equation}\label{eq: pointwise conv}
 h(t, \cdot) \rightarrow C \quad \text{ and } \quad \nabla h(t, \cdot) \rightarrow 0 \quad \text{ in } C(E) \quad \text{ as } t\rightarrow \infty.
\end{equation}
Here $C$ is a constant, $\nabla = (D_1, \dots, D_d)$ is the gradient, and
convergence in $C(E)$ stands for locally uniform convergence in  $E$. In addition to the previous pointwise convergence, we also obtain the
following probabilistic type of convergence: for any fixed $t\geq 0$, as functions of $x\in E$,
\begin{equation}\label{eq: L2 conv}
\begin{split}
 &\expec^{\prob^{\hat{v}, x}} \bra{\int_0^t \nabla h' \barA \nabla h (T-s, X_s) ds} \rightarrow 0 \quad \text{ and } \quad \expec^{\prob^{\hat{v}, x}} \bra{\sup_{0\leq s\leq t} |h(T,x)- h(T-s, X_s)|} \rightarrow 0,
\end{split}
\end{equation}
in $C(E)$ as $T\rightarrow \infty$.
Here, $\nabla h'$ is the transpose of $\nabla h$ and $(\prob^{\hat{v}, x})_{x\in E}$ are probability measures under which the coordinate process $X$ is ergodic (cf. Proposition \ref{prop: e-existence} below).

The Cauchy problem \eqref{eq: pde E} and its ergodic analog \eqref{eq: e-eqn
  E} are closely related to \emph{risk sensitive} control problems of both
finite and infinite horizon: see  \cite{MR1358100, Bensoussan-Frese-Nagai, Nagai-96, Kaise-Sheu} among
others. Indeed, consider
\begin{equation}\label{eq: finite_horizon_rsc}
\max_{z\in\mathcal{Z}}\frac{1}{\theta}\log\left(\espalt{}{}{\exp\left(\theta\left(v_0(X_T) + \int_0^T c(X_s,z_s)ds\right)\right)}\right),
\end{equation}
where $T>0$ represents the horizon, $\theta\in \Real\setminus\cbra{0}$ is the risk-sensitivity parameter, and $\mathcal{Z}$ is a set of acceptable control processes.  For a given $z\in\mathcal{Z}$, $X$ is an $E$-valued diffusion with dynamics $dX_t = b(X_t,z_t)dt + a(X_t)dW_t, X_0 = x$, where $W$ is a $d$-dimensional Brownian motion and $a$ is a matrix such that $aa' = A$.  With $v$ denoting the value function, the standard dynamical programming argument yields the following Hamilton-Jacobi-Bellman (HJB) equation for $v$:
\begin{equation}\label{eq: finite_horizon_rsc_hjb}
\partial_t v = \frac{1}{2}\sum_{i,j=1}^d A_{ij}(x)D_{ij}v + \sup_{z}\cbra{\frac{\theta}{2}\sum_{i,j=1}^d A_{ij}(x)D_i v D_j v + \sum_{i=1}^d b_i(x,z)D_iv + c(x,z)}.
\end{equation}
When $z\mapsto b(x,z)$ is linear and $z\mapsto c(x,z)$ is quadratic the risk-sensitive control problem is called the \emph{linear exponential quadratic problem} %(cf. \cite{Jacobson,Whittle,Bensoussan-Schuppen,Nagai-96})
and the HJB equation reduces to a semilinear equation of type \eqref{eq: pde E}, where the pointwise optimizer $z$ in \eqref{eq: finite_horizon_rsc_hjb} is a linear function of $\nabla v$ and is expected to yield an optimal control.  The long-run analog to \eqref{eq: finite_horizon_rsc} is obtained by maximizing the growth rate:
\begin{equation}\label{eq: infinite_horizon_rsc}
\max_{z\in\mathcal{Z}}\liminf_{T\rightarrow\infty} \frac{1}{\theta T}\log\left(\espalt{}{}{\exp\left(\theta\int_0^T c(X_s,z_s)ds\right)}\right).
\end{equation}
Here, in the linear exponential quadratic case, the solution $(\hat{v},\hat{\lambda})$ from \eqref{eq: e-eqn E} governs both the long-run optimal control and maximal growth rate for \eqref{eq: infinite_horizon_rsc}, while the long-run optimal control is again a linear function of $\nabla \hat{v}$. Thus, the convergence in \eqref{eq: pointwise conv} implies that the optimal control for the finite horizon problem converges to its long-run analog as the horizon goes to infinity.

The convergence in \eqref{eq: pointwise conv} and \eqref{eq: L2 conv} also has direct applications to long-term portfolio choice problems from Mathematical Finance (cf. \cite{MR1675114, MR1790132, MR1882297, MR1802598, MR1910647, MR1932164, MR1995925, Nagai-03, MR2435642} amongst many others). In particular, solutions to \eqref{eq: pde E} and \eqref{eq: e-eqn E} are the value functions for the \emph{Merton} problem where the goal is to maximize expected utility from terminal wealth (finite horizon) or the expected utility growth rate (infinite horizon) for the constant relative risk aversion (CRRA) utility investor in a Markovian factor model.  As in the risk-sensitive control problem, optimal investment policies are governed by  $\nabla v$ and $\nabla \hat{v}$ respectively and hence \eqref{eq: pointwise conv} implies convergence of the optimal trading strategies as the horizon becomes large.  In fact, through the lens of \emph{portfolio turnpikes} (see \cite{guasoni.al.11} and references therein), which state that as the horizon $T$ becomes large, the optimal polices for a \emph{generic} utility function over any finite window $[0,t]$ converge to that of a CRRA utility, the convergence in \eqref{eq: pointwise conv} identifies optimal policies for a wide class of utilities in the presence of a long horizon.  Here, however,  the validity of turnpike results rely upon the convergence in \eqref{eq: L2 conv} instead of \eqref{eq: pointwise conv} (cf. \cite{guasoni.al.11}). As such, \eqref{eq: L2 conv} is essential for proving turnpike results.

In addition to portfolio turnpikes, the convergence in \eqref{eq: L2 conv} implies convergence of solutions to backwards
stochastic differential equations (BSDE) associated to \eqref{eq: pde E} and
\eqref{eq: e-eqn E}.  This connection is made precise in Remark \ref{rem: bsde}, but the
basic idea is that given solutions $v$ to \eqref{eq: pde E} and
$(\hat{v},\hat{\lambda})$ to \eqref{eq: e-eqn E}, for any $T>0$, one can construct BSDE solutions
$(Y^T,Z^T)$ and $(\hat{Y},\hat{Z})$ to \eqref{eq: bsde_fh} and \eqref{eq: bsde_erg} below, respectively. Then, with $\mathcal{Y}^T:= Y^T-\hat{Y}-\hat{\lambda}(T-\cdot)$ and $\mathcal{Z}^T:= Z^T -\hat{Z}$, \eqref{eq: L2 conv} implies
\[
  \lim_{T\rightarrow \infty} \expec^{\prob^{\hat{v}, x}} \bra{\int_0^t \norm{\mathcal{Z}^T_s}^2 ds} =0 \quad \text{and} \quad \lim_{T\rightarrow \infty} \expec^{\prob^{\hat{v}, x}} \bra{\sup_{0\leq s\leq t} \left|\mathcal{Y}^T_s - \mathcal{Y}^T_0\right|}=0, \quad \text{ for any } t>0.
\]

%Risk sensitive control problems naturally arise, and have been studied extensively, in portfolio choice problems in Mathematical Finance, cf. \cite{MR1675114}, \cite{MR1790132}, \cite{MR1882297}, \cite{MR1802598}, \cite{MR1910647}, \cite{MR1932164}, \cite{MR1995925}, \cite{Nagai-03}, and \cite{MR2435642}. Here, both \eqref{eq: pointwise conv} and \eqref{eq: L2 conv} have direct applications. The convergence \eqref{eq: pointwise conv} implies that the finite horizon optimal investment strategy converges to its long-run analogue as the investment horizon increases to infinity. The convergence in \eqref{eq: L2 conv} implies \emph{portfolio turnpikes}, which state that as the investment horizon increases, the optimal investment strategy of the generic utility agent converges to the optimal strategy for the power utility at the beginning of the horizon. Portfolio turnpikes for univariate factor models with constant correlations have been studied in \cite{guasoni.al.11}. However, to establish portfolio turnpikes for multivariate, stochastic correlation factors models, \eqref{eq: L2 conv} is needed. We refer to the accompanying paper \cite{Robertson-Xing} and references therein for detailed discussions on these applications.

In the aforementioned applications, several models for $X$ are widely used. In particular, the \emph{Wishart process} (cf. \cite{Bru} and Example \ref{exa: Wishart} below) has been used for option pricing (cf. \cite{Gourieroux06, Gourieroux09, DaFonseca08, DaFonseca10}) and  portfolio optimization (cf. \cite{buraschi2010correlation, Hata-Sekine}) in multi-variate stochastic volatility models. Wishart processes, taking values in the space of positive definite matrices $\sd$, are multivariate generalizations of the square root Bessel diffusion. They offer modeling flexibility, by allowing stochastic correlations between factors, while still maintaining analytical tractability, by keeping the affine structure. However, the volatility of the Wishart process degenerates on the boundary of $\sd$. Therefore, to include this case, our convergence results need to treat domains other than $\Real^d$ and diffusions with coefficients degenerating on the boundary of the state space.

The convergence  \eqref{eq: pointwise conv} has been obtained via stochastic
analysis techniques. \cite{Nagai-96} and \cite{Nagai-03} study large time
asymptotics when the state space is $\Real^d$ and $A$ may degenerate for large
$|x|$, proving a weak form of the convergence in \eqref{eq: pointwise conv},
i.e., $\lim_{t\rightarrow \infty} h(t,\cdot)/t =0$. In \cite{Ichihara-Sheu}, the convergence in \eqref{eq: pointwise conv} has been obtained when the state space is $\Real^d$ and $A$ is the identity matrix. Even though \cite{Ichihara-Sheu} considers uniformly parabolic equations, by appropriately localizing their arguments, we are able to treat degeneracy on the boundary and replace $\Real^d$ by a general domain $E$. This allows us, in Section \ref{sec: abstract conv}, to develop a general framework to study the large time asymptotics in \eqref{eq: pointwise conv} and \eqref{eq: L2 conv}.  One crucial difference between our treatment and \cite{Ichihara-Sheu} lies in proving the comparison result for solutions to \eqref{eq: pde E}. The uniform parabolic assumption is explicitly used in \cite{Ichihara-Sheu}, and their arguments cannot be extended to the locally parabolic case. We replace the uniform parabolic assumption with an assumption on the Lyapunov function (cf. Assumption \ref{ass: long_run strong} below) used to construct solutions $\hat{v}$ to \eqref{eq: e-eqn E}. Additionally, while existing results focused on convergence \eqref{eq: pointwise conv}, the convergence of type \eqref{eq: L2 conv} was missing in the literature, and in general, does not follow from \eqref{eq: pointwise conv} directly without imposing cumbersome integrability assumptions which are hard to check in general settings.

The general framework presented in Section \ref{sec: abstract conv} gives conditions for convergence in terms of two functions $\phi_0$ and $\psi_0$. Once these two functions satisfy appropriate properties, convergence results in Theorems
\ref{thm: conv} and \ref{thm: pointwise conv} follow. When the state space is
specified, $\phi_0$ and $\psi_0$ provide a channel to explicit convergence
results with assumptions only depending upon the model coefficients. Indeed,
when the state space is $\Real^d$ or $\sd$, growth assumptions on model
coefficients are presented which imply the existence of $\phi_0$ and $\psi_0$, hence the main results (cf. Theorems
\ref{thm: conv Rd} and \ref{thm: conv sd}) readily follow. Though the choice of $\phi_0$ and $\psi_0$ depends upon the state space and model coefficients, the procedures to verify their properties are similar. Therefore the general framework developed in Section \ref{sec: abstract conv} could be applied to other domains as well.

In the rest of the paper, Section \ref{sec: proof sec.2} proves convergence results in Section \ref{sec: abstract conv}. Section \ref{sec: sd proof} verifies results specific to $\Real^d$ and $\sd$. Lastly, Appendix \ref{sec: identification} identifies $\sd$ as a subset of $\Real^{d(d+1)/2}$ which allows us to consider equations with $\sd$-valued spatial variables as special cases of \eqref{eq: pde E} and \eqref{eq: e-eqn E}.

Finally, we summarize several notations used throughout the paper:
\begin{itemize}
\item $\mathbb{M}^d$: the space of $d\times d$ real matrices. For $x\in
  \mathbb{M}^d$, let $x'$ be the transpose of $x$, $\tr(x)$ be the trace of
  $x$, and $\norm{x} = \sqrt{\tr(x'x)}$. For $M,N\in
  \mathbb{M}^d$, the Kronecker product of $M$ and $N$ is denoted by $M\otimes N\in \mathbb{M}^{d^2}$.
\item $\sdall$: the space of $d\times d$ symmetric matrices. $\sd$: the space of $d\times d$ strictly
positive definite symmetric matrices. For $M,N\in \sd$, $M\geq N$ when $M-N$ is positive semi-definite.  Given $M\in\sd$, denote by $\sqrt{M}$ the unique $m\in\sd$ such that $m^2=M$.
\item For regions $E\subseteq\Real^d$ and $F\subseteq\Real^k$ and $\gamma\in (0,1]$ denote by $C^{k,\gamma}(E,F)$ the space of $k$ times differentiable functions whose $k^{th}$ derivative is locally H\"{o}lder continuous with exponent $\gamma$.  Write $C^{k,\gamma}(E)$ for $C^{k,\gamma}(E;\Real)$.
\end{itemize}

\section{Main results}\label{sec: abstract conv}
\subsection{Setup}\label{subsec: setup}
We begin by precisely stating assumptions on the region $E$ as well as the
regularity of the coefficients in \eqref{eq: F E}. As for $E$, assume i)  $E\subseteq \Real^d$ is an open
connected domain star shaped with respect to some $x_0\in E$ \footnote{A domain $E\subset \Real^d$ is star shaped for some $x_0\in E$ if for each $x\in E$ the segment $\{\alpha x_0 + (1-\alpha) x; 0\leq \alpha \leq 1\}$ is contained in E. A convex set is star shaped with respect to any of its points.}; ii) there exist a sequence
$(E_n)_{n\in \Natural}$ of open, bounded, connected domains, each star shaped with
respect to $x_0$ and with $C^{2,\gamma}$ boundary for some $\gamma\in (0,1]$ such that $\bar{E}_n\subset
E_{n+1}$ for each $n$; and iii) $E = \cup_{n} E_n$.

Regarding regularity, for $A_{ij}, \bar{A}_{ij}, B_i,V$, $i,j=1,...,d$, in \eqref{eq: F E}, set $A\dfn (A_{ij})_{i,j=1,...,d},  \bar{A}\dfn(\bar{A}_{ij})_{i,j=1,...,d}$, and $B\dfn (B_i)_{i=1,...,d}$.
Assume $A,\bar{A}\in C^{2,\gamma}(E,\sdall)$, $B\in C^{1,\gamma}(E,\Real^d)$ and $V\in
C^{1,\gamma}(E)$ for some $\gamma\in (0,1]$.

A \emph{classical} solution to \eqref{eq: pde E} is a function $v\in
C^{1,2}((0,\infty) \times E)\cap C([0,\infty)\times E)$ which satisfies \eqref{eq: pde E}. A \emph{classical} solution to
\eqref{eq: e-eqn E} is a pair $(v, \lambda)$ such that $v\in C^2(E)$,
$\lambda\in \Real$, and \eqref{eq: e-eqn E} is satisfied.\footnote{Note that $\fF[\phi]
= \fF[\phi + C]$ for any constant $C$. Hence the first component in solutions
to \eqref{eq: e-eqn E} is only determined up to additive constants.}

The following \emph{local} ellipticity assumptions are imposed on \eqref{eq: pde E} and \eqref{eq: e-eqn E}:
\begin{ass}\label{ass: coeff}
 The functions $A$ and $\barA$ satisfy
 \begin{enumerate}
  \item[i)] For any $n \in \Natural$, $x\in E_n$, and $\xi\in \Real^d$, $\xi' A(x) \xi \geq c_n
    |\xi|^2$, for some constant $c_n>0$;
  \item[ii)] There exist constants $\ok\geq \uk >0$ such that
  \[
   \uk A(x)\leq \barA(x) \leq \ok A(x), \quad \mbox{for all} \quad x\in E.
  \]
 \end{enumerate}
\end{ass}
Let us introduce some more notation which will be used throughout the article.  For a fixed $\phi\in
C^{2,\gamma}(E)$, under the aforementioned domain, regularity and ellipticity assumptions, the
\emph{generalized} martingale problem (cf. \cite{Pinsky}) on $E$ for
\begin{equation}\label{eq: L_phi def}
\cL^\phi := \frac12 \sum_{i,j=1}^d A_{ij} D_{ij} + \sum_{i=1}^d\pare{B_i + \sum_{j=1}^d\barA_{ij} D_j \phi} D_i,
\end{equation}
has a unique solution, denoted by $(\prob^{\phi,x})_{x\in
  E}$.   Here, the probability space is the continuous path space $\Omega = C\left([0,\infty);E\right)$.  The coordinate process is denoted by $X$ so
that $X(\omega)_t = \omega_t$ for $\omega\in\Omega$.  The filtration
$\mathbb{F} = (\F_t)_{t\geq 0}$ is the right-continuous enlargement of the
filtration generated by $X$.
When $\phi\equiv 0$ denote $\cL$ for $\cL^0$. Additionally, as a slight abuse of notation, for a given function $v\in C^{1,2}((0,\infty)\times E)$ and $T>0$, define
\begin{equation}\label{eq: L_phi_def_T}
\cL^{v, T-t} := \frac12 \sum_{i,j=1}^d A_{ij} D_{ij} + \sum_{i=1}^d\pare{B_i + \sum_{j=1}^d\barA_{ij} D_j v(T-t,\cdot)} D_i,\quad 0\leq t \leq T.
\end{equation}
As with the time-homogeneous case,  there exists a
unique solution $(\prob^{v,x}_T)_{x\in E}$ on $(\Omega, \F_T)$ to the generalized martingale problem for
$\cL^{v, T-\cdot}$. Both $(\prob^{\phi,x})_{x\in E}$ and $(\prob^{v,x}_T)_{x\in E}$ satisfy the strong
Markov property.  The martingale problem for $\cL^\phi$ (resp. $\cL^{v, T-\cdot}$) is \emph{well-posed} if the coordinate process does not explode $\prob^{\phi, x}$ a.s. (resp. before $T$, $\prob^{v,x}_T$ a.s) for any $x\in E$.

%The measures $(\prob^{\phi,x})_{x\in E}$ are said to be \emph{recurrent} if for all $x,y\in E$ and $\eps>0$, $\prob^{\phi,x}\bra{\tau_{B_\eps(y)}<\infty} = 1$ where $B_{\eps}(y)$ is the ball of radius $\eps$ centered at $y$ and $\tau_{B_\eps(y)}$ is the first hitting time of the ball. The measures $(\prob^{\phi,x})_{x\in E}$ are said to be \emph{ergodic} if they are recurrent and if there exists a positive function $m \in \mathbb{L}^1(E)$ such that $\tilde{\cL}^{\phi}m = 0$ where $\tilde{\cL}^{\phi}$ is the formal adjoint to $\cL^{\phi}$.

In preparation for the convergence results, let us first establish existence and uniqueness
of classical solutions to \eqref{eq: pde E} and \eqref{eq: e-eqn E}.
For \eqref{eq: e-eqn E},
as in \cite{Kaise-Sheu, Ichihara, Guasoni-Robertson, Ichihara-Sheu}, the following assumption on the \emph{Lyapunov function} helps to construct its solution.

\begin{ass}\label{ass: long_run}

There exists a non-negative $\phi_0 \in C^3(E)$ such that
\begin{equation}\label{eq: barrier function}
\lim_{n\uparrow\infty}\sup_{x\in E\setminus E_n} \mathfrak{F}[\phi_0](x) = -\infty.
\end{equation}

\end{ass}

Given the Lyapunov function $\phi_0$,
the following proposition is a collection of results in
\cite{Kaise-Sheu,Ichihara,Ichihara-Sheu,Guasoni-Robertson}, whose proofs will be discussed briefly in Section \ref{sec: proof sec.2}.

\begin{prop}\label{prop: e-existence}
Let Assumption \ref{ass: coeff} and \ref{ass: long_run} hold. There exists a unique  $\hat{\lambda}\in \Real$ such that the following statements hold:
\begin{enumerate}[i)]
\item There exists a unique (up to an additive constant) $\hat{v}\in C^2(E)$ solving \eqref{eq:
  e-eqn E} with $\hat{\lambda}$ such that
$\left(\prob^{\hat{v},x}\right)_{x\in E}$ is ergodic with an invariant density $\hat{m}$;
\item $\sup_{x\in E}(\hat{v}-\phi_0)(x) < \infty$;
\item $e^{-\uk (\hat{v}-\phi)}\in \bL^1(E,\hat{m})$, for any $\phi\in C^2(E)$
  with $\lim_{n\uparrow \infty} \sup_{x\in E\setminus E_n} \fF[\phi](x)=-\infty$.
\end{enumerate}
\end{prop}

The following assumption enables construction of both super and sub-solutions to \eqref{eq: pde E}, which in turn establishes existence of solutions to \eqref{eq: pde E}.

\begin{ass}\label{ass: wellpose Lphi0}

For $\phi_0$ as in Assumption \ref{ass: long_run}, the martingale problem for $\cL^{\phi_0}$ is well-posed.

\end{ass}

\begin{prop}\label{prop: pde existence}
 Let Assumptions \ref{ass: coeff}, \ref{ass: long_run}, and \ref{ass: wellpose
   Lphi0} hold. For any $v_0$ satisfying
\begin{equation}\label{eq: v0 bound}
\sup_{x\in E} (v_0-\phi_0)(x)<\infty,
\end{equation}
there exists at least one solution $v\in C^{1,2}((0,\infty) \times E) \cap C([0,\infty) \times E)$ solving \eqref{eq: pde E} such that
 \begin{equation}\label{eq: growth v}
 \sup_{(t, x)\in [0,T]\times E} (v(t,x) - \phi_0(x)) <\infty, \quad  \text{ for each } T\geq 0.
 \end{equation}
\end{prop}

The uniqueness of classical solutions to \eqref{eq: pde E} in the class of
functions satisfying \eqref{eq: growth v} follows from the following
comparison result, which requires a strengthening of Assumption \ref{ass:
  long_run}.
%To connect the second condition in \eqref{eq: more phi_0 conditions} below to \eqref{eq: barrier function}, note that for any smooth function $\phi$ and $\delta > 1$, $\mathfrak{F}[\delta\phi] = \delta\mathfrak{F}[\phi] +(\delta-1)((1/2)\delta\nabla\phi'\bar{A}\nabla\phi - V)$.

\begin{ass}\label{ass: long_run strong}
For the $\phi_0$ as in Assumption \ref{ass: long_run},
\begin{equation}\label{eq: more phi_0 conditions}
\begin{split}
\lim_{n\uparrow\infty} \inf_{x\in E\setminus E_n} \phi_0(x) = \infty \quad \text{and} \quad \exists\ \delta > 1\textrm{ such that } \lim_{n\uparrow\infty} \sup_{x\in
  E\setminus E_n}\mathfrak{F}[\delta \phi_0](x) = -\infty.
\end{split}
\end{equation}
\end{ass}

\begin{prop}\label{prop: comparison}
Let Assumptions \ref{ass: coeff}, \ref{ass: long_run}, \ref{ass: wellpose
  Lphi0} and \ref{ass: long_run strong} hold.  Let $v_0,\tilde{v}_0$ satisfy
\eqref{eq: v0 bound} and denote by $v,\tilde{v}$ the respective solutions to
\eqref{eq: pde E} from Proposition \ref{prop: pde existence}. Then
$v_0\leq \tilde{v}_0$ on $E$ implies
\[v\leq \tilde{v}, \quad \text{ on } [0,\infty)\times E.\]
\end{prop}

\subsection{Convergence}\label{subsec: conv}

To study the large time behavior of $v(t, \cdot)$, we restrict the initial condition $v_0$ in \eqref{eq:
  pde E} from the larger class of functions satisfying \eqref{eq: v0 bound}
to the class of functions satisfying
\begin{equation}\label{eq: v0 abs bound}
\sup_{x\in E}\left(|v_0| - \phi_0\right)(x) < \infty.
\end{equation}
Note that $v_0\equiv 0$ satisfies the above bound since $\phi_0\geq 0$. For $v_0$ satisfying \eqref{eq: v0 abs bound}, let $v$ be the unique classical solution to \eqref{eq: pde E} from Proposition \ref{prop: pde existence}.

We define the difference between $v$ and $\hat{\lambda} \cdot +\hat{v}$, where $(\hat{v}, \hat{\lambda})$ comes from Proposition \ref{prop: e-existence}, as
\begin{equation}\label{eq: h}
h(t,x) := v(t,x) - \hat{\lambda} t - \hat{v}(x),\qquad (t,x)\in
[0,\infty)\times E.
\end{equation}
Hence $h\in C^{1,2}((0,\infty)\times E)\cap C([0,\infty)\times E)$ and a direct calculation using \eqref{eq: pde E} and \eqref{eq: e-eqn E} yields
\begin{equation}\label{eq: pde h}
\begin{split}
 \partial_t h = \cL^{\hat{v}}h + \frac12 \nabla h' \barA \nabla h, \quad \text{ on } (0,\infty)\times E,\qquad
 h(0,x)= (v_0 - \hat{v})(x).
\end{split}
\end{equation}

Using \eqref{eq: pde h} and Assumption \ref{ass: coeff}, it follows (cf. equation \eqref{eq: h lb}, Lemma \ref{lem: phi_0 int}, and Remark \ref{rem: uni-local bdd} below) that the functions $\cbra{h(t,\cdot)}_{t\geq 1}$ are bounded from below by an $\hat{m}$ integrable function.  To obtain a corresponding upper bound, crucial for proving convergence, the following assumption is made.  %To gain intuition for the structure of the function $\psi_0$ below, note that in the applications of the next section, modulo constants, $\psi_0 = -\phi_0$.

\begin{ass}\label{ass: psi}
 There exists $\psi_0 \in C^3(E)$ such that
 \begin{align}
  \lim_{n\uparrow \infty} \inf_{x\in E\setminus E_n} \fF[\psi_0](x) = \infty, \qquad & \lim_{n\uparrow \infty} \inf_{x\in E\setminus E_n} (\phi_0 - \psi_0)(x) =\infty; \label{eq: psi_0 hat v}\\
  \inf_{x\in E}\left(\psi_0 + K\phi_0\right) > -\infty, \qquad & \sup_{x\in E} (\fF[\delta\phi_0] + \alpha(\delta\phi_0(x) -
\psi_0)) < \infty, \label{eq: main ub main ub}
 \end{align}
 for some $\alpha, K>0$ and $\delta$ from \eqref{eq: more phi_0 conditions}.
\end{ass}
As shown in \cite[Lemma 4.5]{Ichihara-Sheu}, \eqref{eq: psi_0 hat v} provides a lower bound for $\hat{v}$ in that
\begin{equation}\label{eq: hat v lb}
 \inf_E (\hat{v}- \psi_0) >-\infty.
\end{equation}
Furthermore, \eqref{eq: main ub main ub} provides an upper bound on $h(t, \cdot)$ for $t\geq 0$ (cf. Lemma \ref{lem: main ub} below), which is key for establishing convergence of $h$.  With all the assumptions in place, we now state first convergence result.

\begin{thm}\label{thm: conv}
Let Assumptions \ref{ass: coeff}, \ref{ass: long_run}, \ref{ass: wellpose
  Lphi0}, \ref{ass: long_run strong}, and \ref{ass: psi} hold. Then,  for $v_0$ satisfying \eqref{eq: v0 abs
  bound} and any $t\geq 0$, as functions of $x\in E$,
 \begin{enumerate}
 \item[i)] $\lim_{T\rightarrow \infty} \expec^{\prob^{\hat{v},x}}\bra{\int_0^t (\nabla h)' \overline{A} \nabla h (T-s, X_s) \, ds} =0$ in $C(E)$;
 \item[ii)] $\lim_{T\rightarrow \infty} \expec^{\prob^{\hat{v},x}}\bra{\sup_{0\leq s\leq t} \left|h(T, x) - h(T-s, X_s)\right|} =0$ in $C(E)$.
 \end{enumerate}
\end{thm}

\begin{rem}\label{rem: bsde}

As mentioned in the introduction, convergence in Theorem \ref{thm: conv} can be understood in the context of BSDEs. As generalizations of the
Feynman-Kac formula, solutions to BSDEs provide stochastic representations to solutions of
semi-linear PDEs (cf. \cite{Pardoux-Peng-PDE}). Given $T>0$, a solution
$v$ to \eqref{eq: pde E} and a solution $(\hat{\lambda},\hat{v})$ to
\eqref{eq: e-eqn E} define $(Y^T, Z^T)$ and $(\hat{Y},\hat{Z})$ by
\begin{equation}\label{eq: bsdes}
\begin{split}
(Y^T_t,Z^T_t) &\dfn  (v(T-t, X_t), a'\nabla v(T-t, X_t)),\qquad t\leq T,\\
(\hat{Y}_t,\hat{Z}_t) &\dfn (\hat{v}(X_t), a' \nabla \hat{v}(X_t)),\qquad t\geq 0,
\end{split}
\end{equation}
where $a = \sqrt{A}$. Then, $(Y^T, Z^T)$  solves the \emph{quadratic} BSDE:
 \begin{equation}\label{eq: bsde_fh}
  Y_t = v_0(X_T) + \int_t^T \pare{V(X_u)-\frac{1}{2}(Z^T_u)'M(X_u)Z^T_u} du - \int_t^T(Z^T_u)' dW^T_u,\quad t\leq T.
\end{equation}
Here, $W^T$ is a $\prob^{v,x}_T-$Brownian motion and $M(x)\dfn a^{-1}\barA
a^{-1}(x)$ \footnote{Note that Assumption \ref{ass: coeff}
 $ii)$ implies $\uk
 I_d \leq M\leq \ok I_d$, hence the generator of \eqref{eq: bsde_fh} has
 quadratic growth in $Z$.}. In a similar manner, $(\hat{Y},\hat{Z})$ solves the
 \emph{ergodic} BSDE:
 \begin{equation}\label{eq: bsde_erg}
  \hat{Y}_t = \hat{Y}_s + \int_t^s \pare{V(X_u)-\frac{1}{2}\hat{Z}'_u M(X_u) \hat{Z}_u  -\hat{\lambda}} du -\int_t^s \hat{Z}'_u d\hat{W}_u, \quad \text{ for any } t\leq s,
 \end{equation}
where $\hat{W}$ is a $\prob^{\hat{v},x}-$Brownian motion. This type of ergodic BSDE has been introduced in
 \cite{Fuhrman-Hu-Tessitore} and studied in \cite{Richou},
 \cite{Debussche-Hu-Tessitore}. Now set $\mathcal{Y}^T:= Y^T-
 \hat{Y}-\hat{\lambda}(T-\cdot)$ and $\mathcal{Z}^T:= Z^T-\hat{Z}$. A direct
 calculation using \eqref{eq: h} and \eqref{eq: bsdes} shows
\begin{equation*}
\begin{split}
\int_0^t \norm{\mathcal{Z}^T_s}^2 ds & = \int_0^t \nabla h'A\nabla
h(T-s, X_s)ds;\qquad \mathcal{Y}^T_t - \mathcal{Y}^T_0 = h(T-t,X_t) - h(T,x).
\end{split}
\end{equation*}
Thus, Theorem \ref{thm: conv} and Assumption \ref{ass: coeff} $ii)$ imply
 \[
  \lim_{T\rightarrow \infty} \expec^{\prob^{\hat{v}, x}} \bra{\int_0^t \norm{\mathcal{Z}^T_s}^2 ds} =0 \quad \text{and} \quad \lim_{T\rightarrow \infty} \expec^{\prob^{\hat{v}, x}} \bra{\sup_{0\leq s\leq t} \left|\mathcal{Y}^T_s - \mathcal{Y}^T_0\right|}=0, \quad \text{ for any } t>0.
 \]
\end{rem}

In addition to the convergence in Theorem \ref{thm: conv}, the function $h(t, \cdot)$ and its gradient also converge pointwise as $t\rightarrow\infty$. Such result has been proved in \cite{Ichihara-Sheu} when $E=\Real^d$ and $A= I_d$.

\begin{thm}\label{thm: pointwise conv}
 Let Assumptions \ref{ass: coeff}, \ref{ass: long_run}, \ref{ass: wellpose
  Lphi0}, \ref{ass: long_run strong}, and \ref{ass: psi} hold. Then, for $v_0$ satisfying \eqref{eq: v0 abs
  bound},
 \begin{enumerate}[i)]
  \item[i)] $\lim_{t\rightarrow \infty} h(t, \cdot) = C$ in $C(E)$ for some constant $C$;
  \item[ii)] $\lim_{t\rightarrow \infty} \nabla h(t, \cdot) =0$ in $C(E)$.
 \end{enumerate}
\end{thm}

%---take out the following remark : it will just add to the confusion ---
\nada{

\begin{rem}
As seen in the proofs below, Theorems \ref{thm: conv} and \ref{thm: pointwise conv} do not
explicitly require Assumption \ref{ass: long_run strong}. However, Assumption
\ref{ass: long_run strong} is crucial for both the well-posedness of the
problem (i.e. studying convergence of the unique solution $v$ to \eqref{eq:
  pde v}) and for verifying Assumption \ref{ass: h
  ub}.  Indeed, Proposition \ref{prop: phi_0 int} below shows that Assumption
\ref{ass: long_run strong} immediately yields $\phi_0\in
\bL^1(E,\hat{m})$. Furthermore, Assumption \ref{ass: long_run strong} is a key to identify the upper bound
function $\cJ$ in Assumption \ref{ass: h ub}, cf. Lemma \ref{lem: main ub} below.
\end{rem}

}

%-------end of commenting out ---------

In the next section, Theorems \ref{thm: conv} and \ref{thm: pointwise conv}
are applied to domains $\Real^d$ and $\sd$ respectively.  There,
easy-to-verify growth conditions on coefficients are given so that $\phi_0$ and $\psi_0$ satisfying all requirements are constructed, thus implying the conclusions in Theorems \ref{thm: conv} and \ref{thm: pointwise conv}.

\section{Convergence results when the state space is $\Real^d$ or $\sd$}\label{sec: conv Rd sd}

\subsection{The $\Real^d$ case}\label{subsec: Rd}

This case has been studied in \cite{Ichihara-Sheu} when $A(x)=I_d$. Here, we present an extension when $A$ is locally elliptic. Other than the regularity assumptions at the beginning of Section \ref{sec: abstract conv}, and Assumption \ref{ass: coeff}, the coefficients in $\fF$ satisfy the following growth conditions:

\begin{ass}\label{ass: Rd growth}
$\,$
\begin{enumerate}[i)]
 \item $A$ is bounded and $B$ has at most linear growth. In particular, there exists an $\alpha_1 > 0$ such that $x'A(x)x
   \leq \alpha_1(1+|x|^2)$, for $x\in \Real^d$.
 \item There exist $\beta_1\in \Real$ and $C_1>0$ such that
     \[
      B(x)'x \leq -\beta_1 |x|^2 + C_1, \quad x\in \Real^d.
     \]
 \item There exist $\gamma_1, \gamma_2\in \Real$ and $C_2>0$ such that
     \[
      -\gamma_2 |x|^2 -C_2 \leq V(x) \leq -\gamma_1 |x|^2 + C_2, \quad x\in \Real^d.
     \]
 \item $\max\cbra{\beta_1,\gamma_1} > 0$.  Additionally
\begin{enumerate}[a)]
\item When $\beta_1\leq 0$ and $\gamma_1> 0$ there exist $\alpha_2, C_3>0$ such that
     \[
      x' A(x) x \geq \alpha_2 |x|^2 - C_3, \quad x\in \Real^d;
     \]
 \item When $\beta_1 > 0$ and $\gamma_1< 0$, for the $\alpha_1$ of part $i)$,
\[
\beta_1^2 + 2\gamma_1\ok\alpha_1 > 0.
\]
\end{enumerate}
However, when $\beta_1 >0$ and $\gamma_1\geq 0$, no additional conditions are needed.
\end{enumerate}
\end{ass}

\begin{rem}\label{rem: rd coeff} To understand Assumption \ref{ass: Rd growth} $iv)$, consider a $\Real^d$-valued diffusion $X$ with dynamics
\begin{equation}\label{eq: diffusion}
dX_t = B(X_t)dt + a(X_t)dW_t, \quad X_0=x\in \Real^d,
\end{equation}
where $W$ is a $d$ dimensional Brownian motion and $a= \sqrt{A}$. By Assumption \ref{ass: Rd growth} $i)$ and
the regularity assumptions on $A$ and $B$, \eqref{eq: diffusion} admits a global strong solution $(X_t)_{t\geq 0}$.
If $\beta_1 > 0$ then $X$ is
mean-reverting.  On the other hand, if $\gamma_1 > 0$, $V$ decays to $-\infty$ on the boundary. Thus, part $iv)$ requires either mean reversion or a decaying potential. If both happen, then no additional
parameter restrictions are necessary. However, if mean reversion fails we
require uniform ellipticity for $A(x)$ in the direction of $x$. If $\gamma_1 <
0$ then a delicate relationship between the growth and
degeneracy of $A$, mean reversion of $B$ and the growth of $V$ is needed to ensure convergence results.
\end{rem}

Under these growth assumptions on model coefficients, it follows that with
$\phi_0(x) = (c/2) |x|^2$ and $\psi_0(x) = -(\tilde{c}/2) |x|^2$ for some $c, \tilde{c}>0$, Assumptions \ref{ass: long_run},
\ref{ass: wellpose Lphi0}, \ref{ass: long_run strong}, and \ref{ass: psi}
hold; see Section \ref{subsec: Rd proof} below.  In this case, the main convergence result reads:

\begin{thm}\label{thm: conv Rd}
Suppose that Assumptions \ref{ass: coeff} and \ref{ass: Rd growth} are
satisfied. Then, for any $v_0$ satisfying \eqref{eq: v0 abs bound}, the statements of Theorems \ref{thm: conv} and \ref{thm: pointwise conv} hold.
\end{thm}

\subsection{The $\sd$ case}\label{subsec: sd}
Though $\sd$ cannot be set as $E$ directly, it can be identified with an open set
$E\subset\Real^{d(d+1)/2}$ which is filled up by subregions $E_n$
satisfying the given assumptions. This identification, discussed in detail in Appendix \ref{sec: identification}, allows one to freely go back and forth
between $E$ and $\sd$ and hence results are presented in this section using matrix, rather
than vector, notation.

To define $\fF$ in \eqref{eq: F E} using matrix
notation, note that $\fF$ takes the form
\begin{equation}\label{eq: op gen form}
\fF = \cL + \frac{1}{2}\sum_{i,j=1}^{d}\bar{A}_{ij}D_i D_j +V,
\end{equation}
where the linear operator $\cL$ is given in \eqref{eq: L_phi def} with $\phi
\equiv 0$ and is the generator associated to \eqref{eq: diffusion}.

To define $\cL$ in the matrix setting, we follow the notation used in
\cite[Section 3]{Mayerhofer-Pfaffel-Stelzer}. Let $B: \sd \rightarrow \sdall$ be locally Lipschitz and $F, G:\sd
\rightarrow \mathbb{M}^d$ be such that $G'\otimes
F(x)$ \footnote{Here $\otimes$ is the Kronecker product between two matrices
  whose definition is recalled at the end of Section \ref{sec: intro}.} is
locally Lipschitz.  Consider
\begin{equation}\label{eq: sde sd FG}
 dX_t = B(X_t) dt + F(X_t)dW_t G(X_t) + G(X_t)' dW_t' F(X_t)';\qquad X_0 = x\in\sd,
\end{equation}
where $W= (W^{ij})_{1\leq i,j\leq d}$ is a $\mathbb{M}^d$-valued Brownian
motion. Defining the functions $a^{ij}: \sd \rightarrow \mathbb{M}^d, i,j = 1,...,d$ by
\begin{equation}\label{eq: matrix a}
 a^{ij}_{kl} := F^{ik} G^{lj} + F^{jk}G^{li}, \quad k,l=1,\cdots, d,
\end{equation}
the system in \eqref{eq: sde sd FG} takes the form
\begin{equation}\label{eq: sde sd}
 dX^{ij}_t = B_{ij}(X_t) dt + \tr(a^{ij}(X_t) dW'_t), \quad i,j=1, \dots d.
\end{equation}
Thus $\cL$ is set as the generator associated to $X$:
\begin{equation}\label{eq: L_phi_matrix def}
\cL = \frac12 \sum_{i,j,k,l=1}^d \tr\pare{a^{ij}(a^{kl})'} D^2_{(ij), (kl)}
+  \sum_{i,j=1}^d B_{ij} D_{(ij)},
\end{equation}
where $D_{(ij)} = \partial_{x^{ij}}$ and $D^2_{(ij),(kl)}
= \partial^2_{x^{ij}x^{kl}}$. Now, let $\bar{A}_{(ij),(kl)}, i,j,k,l=1,...,d$
be functions on $\sd$ which are symmetric (in an analogous manner to
$\bar{A}_{ij} = \bar{A}_{ji}$ in the $\Real^d$ case):
\begin{equation}\label{eq: barA cond}
 \overline{A}_{(ij), (kl)} = \overline{A}_{(ji), (kl)} = \overline{A}_{(ij), (lk)} = \overline{A}_{(kl), (ij)}, \quad \text{for } i,j,k,l=1, \cdots, d.
\end{equation}
Given such an $\bar{A}$ and $V: \sd \rightarrow \Real$ the operator
$\mathfrak{F}$ is defined by
\begin{equation}\label{eq: F sd}
 \mathfrak{F} := \cL + \frac12 \sum_{i,j,k,l=1}^d D_{(ij)}\overline{A}_{(ij),(kl)} D_{(kl)}+ V.
\end{equation}

As in Section \ref{sec: abstract conv}, we assume that $\tr(a^{ij}(a^{kl})') \in
C^{2,\gamma}(\sd, \Real)$, $\overline{A}_{(ij), (kl)}\in C^{2,\gamma}(\sd,
\Real)$, $B\in C^{1,\gamma}(\sd, \sdall)$, and $V\in C^{1,\gamma}(\sd,
\Real)$, for some $\gamma\in (0,1]$ and any $i,j,k,l=1, \cdots, d$. The analogue of \eqref{eq: pde E} and \eqref{eq: e-eqn E} are:
\begin{eqnarray}
 \partial_t v &=& \fF[v], \quad (t,x) \in (0,\infty) \times \sd, \quad v(0,x) =v_0(x);\label{eq: pde sd}\\
 \lambda &=& \fF[v], \quad x\in \sd. \label{eq: e-eqn sd}
\end{eqnarray}
The notion of classical solutions to the above equations is defined in the same
manner as in Section \ref{sec: abstract conv}.
Appendix \ref{sec: identification} below shows that equations \eqref{eq: pde
  sd} and \eqref{eq: e-eqn sd} can be treated as special cases of \eqref{eq:
  pde E} and \eqref{eq: e-eqn E}. Hence existence and uniqueness of classical
solutions to \eqref{eq: pde sd} and \eqref{eq: e-eqn sd} follow from
Propositions \ref{prop: e-existence}, \ref{prop: pde existence}, and
\ref{prop: comparison}, provided the requisite assumptions are met.

We now specify Assumption \ref{ass: coeff} to the matrix setting. In particular, the first item below implies that $\fF$ in \eqref{eq: F sd} is locally elliptic; cf. Lemma
\ref{lem: elliptic sd} below. Before stating the assumptions, define
\begin{equation}\label{eq: f_g_def}
f(x):= FF'(x) \quad \text{and} \quad g(x):= G'G(x),\qquad x\in \sd.
\end{equation}
Calculation shows that $\tr\pare{a^{ij}(a^{kl})'} =
f^{ik}g^{jl}+f^{il}g^{jk}+f^{jk}g^{il}+f^{jl}g^{ik}$.
To keep the notation compact, the assumption giving bounds on
$\bar{A}$ below
uses the matrices $a^{ij}$ while all other assumptions use the functions $f$
and $g$.

\begin{ass}\label{ass: coeff sd}
 The functions $f, g$, and $\overline{A}$ satisfy
 \begin{enumerate}
  \item[i)] For any $n\in \Natural$, $x\in E_n \subset \sd$, and $\xi \in \Real^d$, $\xi' f(x) \xi\geq c_n |\xi|^2$ and $\xi' g(x) \xi \geq c_n |\xi|^2$, for some constant $c_n >0$;
  \item[ii)] There exist $\ok\geq\uk >0$ such that, for any $x\in \sd$ and $\theta\in \sdall$,
  \[
   \uk \sum_{i,j,k,l=1}^d \theta_{ij} \tr(a^{ij}(a^{kl})')(x) \theta_{kl} \leq \sum_{i,j,k,l=1}^d \theta_{ij}\overline{A}_{(ij),(kl)}(x)\theta_{kl} \leq \ok \sum_{i,j,k,l=1}^d \theta_{ij} \tr(a^{ij}(a^{kl})')(x)\theta_{kl}.
  \]
 \end{enumerate}
\end{ass}

As in the $\Real^d$ case, growth assumptions on the coefficients are needed to
construct the Lyapunov function. However, unlike $\Real^d$, there are two
types of boundaries to $\sd$ : $\{\|x\|=\infty\}$ and $\{\det(x) =
0\}$. Therefore separate growth assumptions are needed as $x$ approaches each
boundary. Let us first present growth assumptions when $\norm{x}$ is
large. Here, the assumptions are similar to those in Assumption \ref{ass: Rd
  growth} : cf. Remark \ref{rem: rd coeff} for a qualitative explanation of
the restriction in part $iv)$.

\begin{ass}\label{ass: growth sd} There exists  $n_0 > 0$ such that for $\norm{x}\geq n_0$ the following conditions hold:
\begin{enumerate}[i)]
 \item $B$ has at most linear growth and there exist $\alpha_1 > 0$ such that
   $\tr(f(x))\tr(g(x)) \leq \alpha_1 \norm{x}$.
\item  There exist $\beta_1 \in \Real$ and  $C_1>0$ such that
     \[
      \tr(B(x)'x) \leq -\beta_1 \norm{x}^2 + C_1.
     \]
 \item There exist constants $\gamma_1, \gamma_2\in \Real$ and $C_2>0$ such that
     \[
      -\gamma_2 \norm{x} -C_2 \leq V(x) \leq -\gamma_1 \norm{x} + C_2.
     \]
     Furthermore, $V(x)$ is uniformly bounded from above for $\norm{x}\leq n_0$.
 \item $\max\{\beta_1, \gamma_1\} >0$. Additionally
\begin{enumerate}[a)]
\item  When $\beta_1\leq 0$ and $\gamma_1 > 0$, there exists $\alpha_3,C_3>0$
  such that
     \[
       \tr(f(x)xg(x)x) \geq \alpha_3 \norm{x}^3 - C_3.
     \]
\item When $\beta_1 >0$ and $\gamma_1< 0$, for $\alpha_1$ of part $i)$
\begin{equation*}
\beta_1^2 +16 \ok \alpha_1 \gamma_1 >0.
\end{equation*}
\end{enumerate}
However, when $\beta_1 >0$ and $\gamma_1 \geq 0$, no additional conditions are needed.
\end{enumerate}
\end{ass}

For small $\det(x)$, different growth assumptions are needed.  To precisely
state them, for
$\delta\in\Real$ and $x\in \sd$ define
\begin{equation}\label{eq: Hdelta def}
 H_\delta(x) := \tr(B(x) x^{-1}) - (1+ \delta)\,\tr(f(x)x^{-1}g(x)x^{-1}) - \tr(f(x)x^{-1})\, \tr(g(x)x^{-1}).
\end{equation}
The function $H_0$ controls the explosion of solutions to \eqref{eq: sde
  sd}. Indeed, as shown in \cite[Theorem
 3.4]{Mayerhofer-Pfaffel-Stelzer}, \eqref{eq: sde sd} admits a global strong solution when $H_0(x)$ is
 uniformly bounded from below on $\sd$.

\begin{ass}\label{ass: H_eps}
There exits $\epsilon, c_0, c_1 > 0$ such that
\begin{enumerate}[i)]
\item $\inf_{x\in \sd} H_\epsilon(x) > -\infty$.
\item $\liminf_{\det(x)\downarrow 0} \left(H_\epsilon(x) + c_0\log(\det(x))\right) > -\infty$.
\item $\lim_{\det(x)\downarrow 0} \left(H_0(x) + c_1 V(x)\right) = \infty$.
\end{enumerate}
\end{ass}

\begin{rem}\label{rem: global soln}
 Lemma \ref{lem: elliptic sd} below shows that $H_\delta$ is decreasing in
 $\delta$ and hence part $i)$ of Assumption
 \ref{ass: H_eps}  implies $\inf_{x\in\sd} H_0(x)>-\infty$ so that \cite[Theorem
 3.4]{Mayerhofer-Pfaffel-Stelzer} yields the existence of global strong solution
 $(X_t)_{t\in \Real_+}$ to \eqref{eq: sde sd}. Part $ii)$ implies that $\phi_0$ can be chosen (up to additive and multiplicative constants) as $-\log(\det(x))$ when $\det(x)$ is small.
 Since part $ii)$ implies
 $\lim_{\det(x)\downarrow 0}H_0(x) = \infty$,  part $iii)$ allows for the
 potential to decay to $-\infty$
 as $\det(x)\downarrow 0$ but at a rate slower than the rate at which $H_0$ goes to $\infty$.
\end{rem}

\begin{exa}\label{exa: Wishart}
 The primary example for \eqref{eq: sde sd} is when $X$ follows a Wishart process:
 \begin{equation}\label{eq: wishart sde}
  dX_t = (LL' + K X_t + X_t K') \, dt + \sqrt{X_t} dW_t \Lambda' + \Lambda dW'_t \sqrt{X_t},
 \end{equation}
where $K, L, \Lambda \in \mathbb{M}^d$ with $\Lambda$ invertible. Here, $f$
and $g$ from \eqref{eq: f_g_def} specify to $f(x) = x$ and $g(x) =
\Lambda\Lambda'$. Thus, part $i)$ of Assumption \ref{ass: coeff sd} as well as
parts $i),ii)$ of Assumption \ref{ass: growth sd} readily follow. $H_\delta$ from \eqref{eq: Hdelta def} takes the form $H_\delta(x) = \tr((LL' -
(d+1+\delta)\Lambda \Lambda')x^{-1}) + 2\tr(K)$. Then $LL'\geq (d+1) \Lambda
\Lambda'$ ensures that $H_0$ is uniformly bounded from below on
$\sd$, and hence \eqref{eq: wishart sde} admits a unique global
strong solution. However, the slightly
stronger assumption: $LL' > (d+1) \Lambda \Lambda'$, is needed to satisfy
Assumption \ref{ass: H_eps}.  Indeed, for $LL' > (d+1)\Lambda\Lambda'$, part
$i)$ of Assumption \ref{ass: H_eps} is evident, and part $ii)$ holds because, as $\det(x)\downarrow 0$,
$\tr(Cx^{-1})+\log(\det(x))\rightarrow\infty$ for any $C\in \sd$.
Lastly, any potential $V$ which is bounded from below by $-\tr((LL'-(d+\delta+1)\Lambda\Lambda')x^{-1})$, for some $\delta>0$ and small $\det(x)$ satisfies part $iii)$.
\end{exa}

Let $n_0$ be from Assumption \ref{ass: growth sd} and let
$\overline{c},\underline{c}, C > 0$ be constants.  Under Assumptions
\ref{ass: growth sd} and \ref{ass: H_eps} a candidate Lyapunov function
$\phi_0$ is given by
\begin{equation}\label{eq: phi0 matrix def}
\phi_0(x) := -\underline{c}\log(\det(x)) + \overline{c}\norm{x}\eta(\norm{x}) + C,
\end{equation}
where the cutoff function $\eta\in C^{\infty}(0,\infty)$ is such that $0\leq \eta\leq 1$,
$\eta(x)=1$ when $x>n_0+2$, and $\eta(x) =0$ for $x < n_0 + 1$.
Furthermore, for $\underline{k}, \overline{k}>0$, $\psi_0$ is chosen as
\[
 \psi_0 (x):= \underline{k} \log(\det(x)) - \overline{k} \norm{x} \eta(\norm{x}), \quad x\in \sd.
\]
Section \ref{subsec: sd proof} proves that, under Assumptions \ref{ass: coeff sd} - \ref{ass: H_eps}, there exist $\underline{c}$, $\overline{c}$, $C$, $\underline{k}$, and $\overline{k}$ such that Assumptions \ref{ass: long_run}, \ref{ass: wellpose
  Lphi0}, \ref{ass: long_run strong}, and \ref{ass: psi} are satisfied. Then the main convergence result in the $\sd$ case readily follows:

\begin{thm}\label{thm: conv sd}
Suppose that Assumptions \ref{ass: coeff sd}, \ref{ass: growth sd}, and \ref{ass: H_eps} are satisfied. Then, for any $v_0$ satisfying  \eqref{eq: v0 abs bound}, the statements of Theorems \ref{thm: conv} and \ref{thm: pointwise conv} hold.
\end{thm}

\section{Proofs in Section \ref{sec: abstract conv}}\label{sec: proof sec.2}

\subsection{Proofs in Section \ref{subsec: setup}}
Let us first briefly discuss proofs for Propositions \ref{prop: e-existence} and \ref{prop: pde existence}. Proposition \ref{prop: e-existence} i) essentially follows from \cite[Theorems 13, 18]{Guasoni-Robertson}, with only the following minor modifications. First, in \cite{Guasoni-Robertson} it is assumed that $\sup_{x\in E}V(x) < \infty$ and that $\bar{A}(x)$ takes a particular form.  However, $\sup_{x\in E}V(x) < \infty$ is not actually necessary in the presence of Assumption \ref{ass: long_run} and the only essential fact used regarding $\bar{A}$ (labeled $\hat{A}$ therein) is that Assumption \ref{ass: coeff} holds: see equation (91) therein.  To see this, when repeating the proof of Theorem 13 on page 272 of \cite{Guasoni-Robertson} note that since $\sup_{x\in E}\mathfrak{F}[\phi_0](x) < \infty$, it follows that $\mathfrak{F}[\phi_0] - \lambda < 0$ on $E$ for sufficient large $\lambda$. Then, since the generalized principal eigenfunction for the operator $L^c$ therein with $c=\uk$ is finite (as can be seen by repeating the argument on page 272), it follows again that for $\lambda$ large enough there exist strictly positive solutions $g$ of $L^cg = \lambda g$, at which point setting $f = (1/c)\log(g)$ and using Assumption \ref{ass: coeff} it follows that $\mathfrak{F}[f] - \lambda > 0$ on $E$.  From here the result follows exactly as in \cite[Theorem 13]{Guasoni-Robertson}.  The proof of \cite[Theorem 18]{Guasoni-Robertson} follows with only notational modifications.

To prove Proposition \ref{prop: e-existence} $ii)$, calculation shows that under Assumption \ref{ass: coeff}, for any two $\phi, \psi\in C^2(E)$,
the function $w:= e^{-\underline{\kappa} (\psi-\phi)}$ satisfies
\begin{equation}\label{eq: ineq ichihara}
\mathcal{L}^{\psi} w \leq \underline{\kappa} w (\mathfrak{F}[\phi] - \mathfrak{F}[\psi]) \quad  \text{on } E,
\end{equation}
where $\mathcal{L}^\psi$ is defined in \eqref{eq: L_phi def}. This is exactly \cite[Lemma 4.2 (b)]{Ichihara}\footnote{Note a negative sign needs to added to $\phi$ and $\psi$ in \cite{Ichihara} to fit our context.} with $\epsilon=0$ in (A5) therein. Then repeating remaining arguments in \cite[Section 4]{Ichihara}, the statement follows from \cite[Theorem 2.2 $(i)\Rightarrow (iv)$]{Ichihara}.

Define the stopping times $\cbra{\tau_n}_{n\in\Natural}$ as the first exit time of $X$ from $E_n$:
\begin{equation}\label{eq: tau_n_def}
\tau_n \dfn \inf\cbra{ t\geq 0 : X_t\not\in E_n}.
\end{equation}
Proposition \ref{prop: e-existence}
$iii)$ essentially follows from \cite[Proposition 2.4]{Ichihara-Sheu}. To
connect to the proof therein, note that \eqref{eq: ineq ichihara} with $\psi= \hat{v}$ and  $x\leq \max\{x+1,0\}$ combined yield:
\begin{equation*}
\cL^{\hat{v}}w \leq \uk w(\mathfrak{F}[\phi]-\hat{\lambda}) \leq \uk
w\left(\max\{\mathfrak{F}[\phi]-\hat{\lambda} + 1,0\}\right).
\end{equation*}
Since $\mathfrak{F}[\phi](x)\rightarrow -\infty$ as $x\rightarrow \partial E$
there is a constant $M$ so that $\cL^{\hat{v}}w \leq M$ on $E$.  Thus, by first
  stopping at $\tau_n$ and then using Fatou's lemma, there is a constant $C = C(x)$
  such that
$\expec^{\prob^{\hat{v},x}}\bra{e^{-\uk(\hat{v}-\phi)(X_t)}} \leq C+ Mt$.
The result now follows by repeating the argument in \cite[Proposition
2.4]{Ichihara-Sheu} starting right after equation $(2.4)$ therein.

Proposition \ref{prop: pde existence} is proved by first constructing super-
and sub-solutions $\psi_1$ and $\psi_2$ to \eqref{eq: pde E} and then repeating
the arguments in \cite[Theorems 3.8, 3.9]{Ichihara-Sheu}. Even though the
equation is uniformly parabolic in \cite{Ichihara-Sheu}, the solution $v$ is
constructed, using the given super- and sub-solutions, via a sequence of localized problems, each of which is uniformly
parabolic, cf. \cite[Equation (3.6)]{Ichihara-Sheu}. Here, the sequence of
localized problems can be considered  on $(E_n)_{n\in \Natural}$, where $A$ is
uniformly elliptic in each $E_n$ due to Assumption \ref{ass: coeff}
$i)$.  To construct the super- and sub-solutions $\psi_1$
and $\psi_2$, for $\zeta > 0$, define
\begin{equation*}
\psi(t,x;\zeta) \dfn \phi_0(x) +
\frac{1}{\zeta}\log\left(\expec^{\prob^{\phi_0,x}}\bra{\exp\left(\zeta(v_0 -
      \phi_0)(X_t) +\zeta\int_0^t \mathfrak{F}[\phi_0](X_s)ds\right)}\right).
\end{equation*}
In view of Assumptions \ref{ass: long_run}, \ref{ass: wellpose Lphi0} and
equation \eqref{eq: v0 bound} it follows that $\psi(t,x;\zeta)$ is
well-defined and finite for $(t,x)\in [0,\infty)\times E$.  With $\psi_1 =
\psi(\cdot ; \ok)$ and $\psi_2 = \psi(\cdot; \uk)$, H\"{o}lder's inequality implies $\psi_2\leq \psi_1$. Moreover, one can check that $\psi_1$
and $\psi_2$ are super- and sub-solutions of \eqref{eq: pde E} respectively.
This fact follows from  the extension of the classical Feynman-Kac formula to
the current, locally elliptic, setup; see
\cite{Heath-Schweizer,guasoni.al.11}. Thus, Proposition \ref{prop: pde
  existence} holds.

Now we prove Proposition \ref{prop: comparison} which does not follow from \cite[Theorem 3.6]{Ichihara-Sheu}. Let us first prepare a prerequisite result.
%Recall notation at the end of Section \ref{sec: intro} and define $\sd$-valued function $a$ by $a(x) = \sqrt{A}(x),x\in E$.
\begin{lem}\label{lem: P^T wellposed}
 Let Assumptions \ref{ass: coeff}, \ref{ass: long_run}, and \ref{ass: wellpose
  Lphi0} and \ref{ass: long_run strong} hold.  Let $v$ be a classical solution to \eqref{eq: pde E} in Proposition
  \ref{prop: pde existence} with initial condition $v_0$ satisfying \eqref{eq: v0 bound}. Then, for any $T>0$, the martingale
problem for $\cL^{v, T-\cdot}$ on $E$ is well-posed. Hence the coordinate process does not hit the boundary of $E$ before $T$, $\prob^{v,x}_T$ a.s., for any $x\in E$.
\end{lem}

\begin{proof}
Set $\tilde{v}(t,x) = v(t,x) - \delta\phi_0(x)$ for $(t,x) \in [0,T]\times E$,
where $\delta$ is from \eqref{eq: more phi_0 conditions}. It follows from
\eqref{eq: growth v} and \eqref{eq: more phi_0 conditions} that
 \begin{equation}\label{eq: tildev growth 1}
  \lim_{n\uparrow \infty} \sup_{(t,x) \in [0,T]\times E\setminus E_n}
  \tilde{v}(t,x) = \lim_{n\uparrow \infty} \sup_{(t,x) \in [0,T]\times E\setminus E_n} (v(t,x)-\phi_0(x) - (\delta-1)\phi_0(x))  =-\infty.
 \end{equation}
A direct calculation shows (note: $\partial_t \tilde{v} = \partial_t v$)
 \[
  \cL^{v, T-\cdot} \tilde{v} = \partial_t\tilde{v} -
  \mathfrak{F}[\delta\phi_0] + \frac{1}{2}\left(\nabla
    v-\delta\nabla\phi_0\right)'\bar{A}\left(\nabla
    v-\delta\nabla\phi_0\right) \geq \partial_t\tilde{v} - \mathfrak{F}[\delta\phi_0].
 \]
Since \eqref{eq: more phi_0 conditions} assumes  $\lim_{n\uparrow \infty}\sup_{x\in E\setminus E_n} \fF[\delta\phi_0] =-\infty$, there exists a constant $C$ such that
\begin{equation}\label{eq: Lv lb}
 -\partial_t \tilde{v} + \cL^{v, T-\cdot} \tilde{v} \geq C, \quad \text{on } (0,T]\times E.
\end{equation}
The well-posedness of the martingale problem for $\cL^{v,T-\cdot}$ on $E$ now
follows from \cite[Theorem 10.2.1]{MR2190038}, by defining $\phi_T(t,x)
\dfn -\tilde{v}(T-t,x) + K$ for some $K$ so that $\phi_T(t,x) \geq 1,
(t,x)\in[0,T]\times E$.  Such a $K$ exists in view of \eqref{eq: tildev growth
  1}. Note also that the coefficients $a_n,b_n$ in \cite[Theorem 10.2.1]{MR2190038} can easily be
constructed in the present setup, cf. \cite[p.250]{MR2190038}, and $\lambda$ there can be chosen as any positive constant larger than $-C$.

\end{proof}

\begin{proof}[Proof of Proposition \ref{prop: comparison}]
For the given $\tilde{v}_0 \geq v_0$ and associated solutions
$\tilde{v},v$  in Proposition \ref{prop: pde existence}, fix a $T>0$ and set $w(t,x) = \tilde{v}(T-t,x) - v(T-t,x)$, for
$t\leq T$ and $x\in E$. Since $\tilde{v},v$ solve the differential
expression in \eqref{eq: pde E} it follows
that
\[\partial_t w + \cL^{v, T-\cdot} w = -(1/2)\nabla w'\bar{A}\nabla w.\]
Then under $\prob^{v,x}_T$, which is the solution to the martingale problem for $\cL^{v, T-\cdot}$ in Lemma \ref{lem: P^T wellposed}, we have
\begin{equation*}
\uk \left(w(T,X_T) - w(0,x)\right) \leq \uk\int_0^T\nabla w'a (s,X_s)dW^{v}_s - \frac{1}{2}\uk^2\int_0^T\nabla w'A\nabla
w(s,X_s)dx,
\end{equation*}
where $W^v$ is a $\prob^{v,x}_T$-Brownian Motion and the inequality follows from $\barA \geq \uk A$. Exponentiating both sides of the previous inequality and taking $\prob^{v, x}_T$-expectations, we obtain
\begin{equation*}
e^{-\uk w(0,x)}\espalt{\prob^{v,x}_T}{}{e^{\uk w(T,X_T)}} \leq
\espalt{\prob^{v,x}_T}{}{\mathcal{E}\left(\uk\int_0^\cdot \nabla w'a(s,X_s)dW^v_s\right)_T} \leq 1.
\end{equation*}
Plugging in for $w= \tilde{v}-v$ and using $\tilde{v}_0 \geq v_0$ gives
\begin{equation*}
1 \geq
e^{-\uk(\tilde{v}(T,x)-v(T,x))}\espalt{\prob^{v,x}_T}{}{e^{\uk(\tilde{v}_0-v_0)(X_T)}}
\geq e^{-\uk(\tilde{v}(T,x)-v(T,x))},
\end{equation*}
which confirms the assertion since $\uk > 0$.
\end{proof}

\subsection{Proofs in Section \ref{subsec: conv}}\label{sec: conv proof}
Theorems \ref{thm: conv} and \ref{thm: pointwise conv} are proved in this section.
For $\hat{v}$ in Proposition \ref{prop: e-existence} and $x\in E$, to simplify notation, we denote
\[
\hat{\prob}^x := \prob^{\hat{v},x} \quad \text{ and } \quad \hat{\expec}^x := \expec^{\prob^{\hat{v},x}}.
\]
Throughout this section $C$ is a universal constant which may be different in different places and the assumptions of Theorem \ref{thm: conv} are enforced. In particular, $v_0$ is chosen to satisfy \eqref{eq: v0 abs bound}. The following facts regarding ergodic diffusions are used repeatedly
throughout the sequel:

\begin{rem}[Ergodic results]\label{rem: uni-local bdd}
Recall from Proposition \ref{prop: e-existence} $i)$ yields that $X$ is ergodic under $(\hat{\prob}^x)_{x\in E}$ with invariant density $\hat{m}$. Given a continuous non-negative function $f$ such that $f\in \bL^1(E,\hat{m})$, \cite{Pinchover-92} and \cite[Corollary 5.2]{Pinchover-04} prove
\begin{enumerate}[i)]
\item $\hat{\expec}^x [f(X_t)]<\infty$ for any $x\in E$ and $t>0$;
\item $\sup_{t\geq \delta}\sup_{x\in E_n} \hespalt{x}{}{f(X_t)} < \infty$ for any $\delta > 0$ and integer $n$;
\item $\lim_{t\rightarrow \infty} \hat{\expec}^x[f(X_t)] = \int_E f(x) \hat{m}(x)\, dx$ in $C(E)$.
\end{enumerate}
\end{rem}

To prove Theorems \ref{thm: conv} and \ref{thm: pointwise conv}, we first prepare several results.

\begin{lem}\label{lem: phi_0 int}
For $\phi_0$ in Assumption \ref{ass: long_run} and $\hat{m}$ in Proposition \ref{prop: e-existence} $i)$,
$\phi_0\in \bL^1(E,\hat{m})$.
\end{lem}

\begin{proof}
Set $\tilde{\phi}_0 \dfn \delta\phi_0$.  From
Proposition \ref{prop: e-existence} $ii)$,
$\uk(\delta-1)\phi_0 = \uk(\tilde{\phi}_0 - \phi_0) \leq \uk(\tilde{\phi}_0 -
\hat{v}) + \uk C = -\uk(\hat{v}-\tilde{\phi}_0) + \uk C$ for some $C>0$.  Then $e^{\uk(\delta-1)\phi_0}\in \bL^1(E,\hat{m})$ follows Proposition \ref{prop:
  e-existence} $iii)$ and Assumption \ref{ass: long_run strong}. Since $\phi_0$ is non-negative, then the statement is confirmed.
\end{proof}

\begin{cor}\label{cor: ui}
Let $x\in E$, $0\leq t\leq T$ and $\cbra{\tau_n}_{n\in\Natural}$ be as in \eqref{eq: tau_n_def}. Then the family of random variables
\begin{equation*}
\{h(T-t\wedge\tau_n, X_{t\wedge\tau_n});\, n \in \Natural\},
\end{equation*}
is $\hat{\prob}^x$-uniformly integrable.

\end{cor}

\begin{proof}
Applying Ito's formula to $h(t-\cdot, X_\cdot)$ and utilizing both \eqref{eq:
  pde h} and $\barA\geq\uk A$, we obtain, for any stopping time $\tau$ for
which $\tau \leq t$,
\begin{equation*}\label{eq: h ito}
\begin{split}
 \uk h(t-\tau, X_\tau) &\leq \uk h(t, x) - \frac{\uk^2}{2} \int_0^\tau \nabla h' A \nabla h(t-u, X_u)
\,du + \int_0^\tau \uk \nabla h'a(t-u, X_u) \, d\hat{W}_u,
\end{split}
\end{equation*}
where $\hat{W}$ is a $\hat{\prob}^{x}-$Brownian motion. Exponentiating both sides of the
previous inequality and taking expectations gives
\begin{equation}\label{eq: h lb exp ub}
\hat{\expec}^x\bra{e^{\uk h(t-\tau,X_\tau)}} \leq e^{\uk
  h(t,x)}\hat{\expec}^x\bra{\mathcal{E}\left(\uk\int_0^\cdot\nabla
    h'a(t-u,X_u)d\hat{W}_u\right)_\tau}\leq e^{\uk h(t,x)},
\end{equation}
and thus, at $\tau = s$ for the fixed time $s\leq t$:
\begin{equation}\label{eq: h lb exp}
\frac{1}{\uk} \log \hat{\expec}^x \bra{e^{\uk h(t-s, X_s)}}\leq h(t,x).
\end{equation}
Proposition \ref{prop: e-existence} $ii)$ and \eqref{eq: v0 abs bound} imply
both $\hat{v}\leq \phi_0 + C$ and $v_0\geq -\phi_0 - C$. Thus, \eqref{eq: h lb exp} with $s=t$ and Jensen's inequality combined imply
\begin{equation}\label{eq: h lb}
h(t,x)\geq \frac{1}{\uk}\log\hespalt{x}{}{e^{\uk(v_0-\hat{v})(X_t)}}\geq
\hespalt{x}{}{(v_0-\hat{v})(X_t)} \geq -2\hespalt{x}{}{\phi_0(X_t)} - C,
\end{equation}
for some constant $C$. Therefore, with $h_-\dfn \max\{-h,0\}$, the Markov property and $\phi_0\geq 0$ combined yield
\[
h_-(T-t\wedge\tau_n, X_{t\wedge\tau_n}) \leq C+ 2\hespalt{X_{t\wedge\tau_n}}{}{\phi_0(X_{T-t\wedge\tau_n})} = C + 2 \hat{\expec}^x[\phi_0(X_T) \,|\, \F_{t\wedge \tau_n}],
\]
By Lemma \ref{lem: phi_0 int} and Remark \ref{rem: uni-local bdd} $i)$, we
have $\hespalt{x}{}{\phi_0(X_T)} < \infty$. Thus, the random variables
$\{h_-(T-t\wedge \tau_n, X_{t\wedge \tau_n}); n=1,2,...\}$ are uniformly integrable under $\hat{\prob}^x$.

As for the positive part, set $h_+\dfn\max\{h,0\}$. Since for any constant
$k>0$, $e^{kh_+} \leq 1 + e^{kh}$, \eqref{eq: h lb exp ub} implies there is a
$C = C(T,x) > 0$ so that
\begin{equation*}
\hespalt{x}{}{e^{\uk h(T-t\wedge\tau_n,X_{t\wedge\tau_n})_+}} \leq C, \quad \text{ for any } n.
\end{equation*}
The uniform integrability of $\{h(T-t\wedge\tau_n,X_{t\wedge\tau_n}); n\in \Natural\}$ now follows, finishing the proof.
\nada{
\begin{equation*}
h(T-t\wedge\tau_n,X_{t\wedge\tau_n}) \leq h(T,x) + M^T_{t\wedge\tau_n},
\end{equation*}
where $M^T_\cdot = \int_0^\cdot \nabla
h'a(T-u,X_u)\,d\hat{W}_u$. Now
\begin{equation*}
\begin{split}
\hespalt{x}{}{(M^T_{t\wedge\tau_n})^2} =& \hespalt{x}{}{\int_0^{t\wedge\tau_n}\nabla h'A\nabla h
  (T-u,X_u)du}\leq
\frac{2}{\uk}\,\hespalt{x}{}{\frac{1}{2}\int_0^{t\wedge\tau_n}\nabla
  h'\bar{A}\nabla h(T-u,X_u)du}\\
=&\frac{2}{\uk}\left(h(T,x) - \hespalt{x}{}{h(T-t\wedge\tau_n,X_{t\wedge\tau_n})}\right)\leq \frac{2}{\uk}\left(h(T,x) + \hespalt{x}{}{h_-(T-t\wedge\tau_n,X_{t\wedge\tau_n})}\right),
\end{split}
\end{equation*}
where $\barA \geq \uk A$ is applied to obtain the first inequality.
Now, $h_-(T-t\wedge\tau_n,X_{t\wedge\tau_n})$ is $\hat{\prob}^x$ uniformly integrable, in particular, $\sup_n\hespalt{x}{}{h_-(T-t\wedge\tau_n,X_{t\wedge\tau_n})} <\infty$.
This gives that $\sup_n
  \hespalt{x}{}{(M^T_{t\wedge\tau_n})^2} < \infty$ which in turn
  implies the $\hat{\prob}^x$ uniform integrability of $h_+(T-t\wedge\tau_n,
  X_{t\wedge\tau_n})$ and finishes the proof.}
\end{proof}

The next result identifies an upper bound on $h(t, \cdot)$, uniformly in $t\geq 0$.  The statement and proof are similar to \cite[Lemma 4.7]{Ichihara-Sheu}.

\begin{lem}\label{lem: main ub}
Let $\cJ(x) \dfn J\left(1+\phi_0(x) + \hat{v}_{-}(x)\right)$, for $x\in E$. Here $\hat{v}_- := \max\{-\hat{v}, 0\}$.
Then $\cJ\in C(E,\Real)\cap \mathbb{L}^1(E,\hat{m})$ and there exists a sufficiently large constant $J$ such that
\begin{equation}\label{eq: h ub}
\sup_{t\geq 0} h(t,x) \leq \cJ(x),\quad x \in E.
\end{equation}

\end{lem}

\begin{proof}
Due to \eqref{eq: hat v lb} and the first inequality in \eqref{eq: main ub main ub}, $\hat{v}_{-} \leq C -\psi_0 \leq C + K\phi_0$, hence $\cJ\in \mathbb{L}^1(E,\hat{m})$ follows from Lemma \ref{lem: phi_0 int}. Moreover it is clear that $\cJ\in C(E,\Real)$.

Let us prove \eqref{eq: h ub}. Since $v_0$ satisfies \eqref{eq: v0 abs bound}, $v_0
\leq \phi_0 + C$ for some constant $C$. Thus, by the comparison principle in
Proposition \ref{prop: comparison} it suffices to prove \eqref{eq: h ub} when $v_0 = \phi_0 + C$.  Additionally, since $\mathfrak{F}[v+C] = \mathfrak{F}[v]$ for
any constant $C$, one can set $v_0 = \phi_0$ without loss of generality. Thus,
let $v$ be the solution of \eqref{eq: pde E} with initial condition $\phi_0$
and let $h(t,x) = v(t,x) - \hat{\lambda}t -\hat{v}(x)$.

Set $w(t,x) \dfn \delta \phi_0(x) + \hat{\lambda} t - v(t,x)$. We first derive
upper and lower bounds for $w$. On the one hand, note that $w(t,x) = \delta
\phi_0(x) -\hat{v}(x) + (\hat{v}(x) + \hat{\lambda} t - v(t,x))$ and that
$\hat{v}(x) + \hat{\lambda} t$ satisfies  \eqref{eq: pde E} with the initial
condition $\hat{v}$. Proposition \ref{prop: e-existence} $ii)$ and
$\phi_0=v_0$ give $\hat{v}(x) \leq   v_0(x) + C$ and hence a second
application of  Proposition
\ref{prop: comparison}  yields $\hat{v}(x) + \hat{\lambda} t \leq v(t,x) + C$ on $[0,T]\times E$. Thus
\begin{equation}\label{eq: w ub}
 w(t,x) \leq \delta \phi_0(x) - \hat{v}(x) + C,  \quad \text{ on } [0,T]\times E.
\end{equation}
On the other hand, \eqref{eq: growth v} implies the existence of constant $C_T$, which may depend on $T$, such that $v(t,x) \leq \phi_0(x) + C_T$ on $[0,T]\times E$. Then \begin{equation}\label{eq: w lb}
w(t,x) \geq (\delta-1) \phi_0(x) + \hat{\lambda} t - C_T \geq \tilde{C}_T, \quad \text{ on } [0,T]\times E,
\end{equation}
for some constant $\tilde{C}_T$, where the second inequality follows from $\delta>1$ and $\phi_0 \geq 0$.

A direct calculation shows
$
\cL^{v, T-\cdot}w = \mathfrak{F}[\delta\phi_0] + w_t - \hat{\lambda} -
(1/2)\nabla w'\bar{A}\nabla w
$, which implies $-w_t + \cL^{v, T-\cdot} w \leq \fF[\delta \phi_0] - \hat{\lambda}$. For the given $\alpha$ in \eqref{eq: main ub main ub}, applying Ito's formula to $e^{\alpha \cdot} w(T-\cdot, X_\cdot)$ and utilizing the previous inequality, we obtain for each $n$ (recall $\tau_n$ from \eqref{eq: tau_n_def}):
\[
\espalt{\prob^{v,x}_{T}}{}{e^{\alpha (T\wedge
    \tau_n)}w(T-T\wedge\tau_n,X_{T\wedge\tau_n})}\leq w(T,x) +
\espalt{\prob^{v,x}_{T}}{}{\int_0^{T\wedge\tau_n}e^{\alpha
    s}(\fF[\delta\phi_0] - \hat{\lambda} + \alpha w)(T-s,X_s)ds},
\]
Since $w$ is bounded from below (cf. \eqref{eq: w lb}), applying Fatou's lemma on the left-hand-side yields
\begin{equation*}\label{eq: square temp1}
\begin{split}
e^{\alpha T}\expec^{\prob^{v,x}_T}\bra{w(0, X_T)}\leq w(T,x) + \liminf_{n\uparrow\infty}\left(\espalt{\prob^{v,x}_{T}}{}{\int_0^{T\wedge\tau_n}e^{\alpha
    s}(\fF[\delta \phi_0] - \hat{\lambda} + \alpha
    w)(T-s,X_s)ds}\right).
\end{split}
\end{equation*}
On the right-hand-side, \eqref{eq: main ub main ub} implies $M:= \sup_{x\in E} (\fF[\delta\phi_0] + \alpha(\delta\phi_0(x) -
\psi_0))<\infty$. Therefore \eqref{eq: w ub} and \eqref{eq: hat v lb} combined yield
\begin{equation*}\label{eq: square temp3}
\fF[\delta \phi_0]- \hat{\lambda} + \alpha w \leq \fF[\delta
\phi_0] - \hat{\lambda} + \alpha(\delta\phi_0- \hat{v}+C) \leq \fF[\delta \phi_0] - \hat{\lambda} + \alpha(\delta \phi_0 -\psi_0) + C
\leq M - \hat{\lambda} + C.
\end{equation*}
Set $\hat{M} = \max\{M-\hat{\lambda}+C,0\}/\alpha$.
Combining the previous two inequalities and using $w(0, x)= \delta \phi_0(x)- v_0(x) = (\delta -1) \phi_0(x)$, we obtain
\begin{align*}
(\delta -1) e^{\alpha T}\espalt{\prob^{v,x}_{T}}{}{\phi_0(X_T)} &\leq w(T,x) + \hat{M} (e^{\alpha T}-1)\leq e^{\alpha T}\left(\delta\phi_0(x) + \hat{v}_-(x) + C+\hat{M}\right),
\end{align*}
where the second inequality follows from \eqref{eq: w ub}. Thus, by taking $C > 0$ sufficiently large,
\begin{equation}\label{eq: pTx phi0 expec ub}
\espalt{\prob^{v,x}_{T}}{}{\phi_0(X_T)} \leq C(1 + \phi_0(x) +\hat{v}_{-}(x)), \quad \text{ for all } x\in E \text{ and } T\geq 0.
\end{equation}

Calculation shows that $h$ satisfies $h_t = \cL^{v,T-\cdot}h - (1/2)\nabla
h'\bar{A}\nabla h \leq \cL^{v,T-\cdot}h - (1/2)\uk\nabla h'A\nabla h$, since
$\bar{A}\geq \uk A$. Applying Ito's formula to $\uk h(T-\cdot, X_\cdot)$ yields
 \[
 \begin{split}
  \uk\left(\expec^{\prob^{v,x}_T}[h(0, X_T)] - h(T, x)\right)&\geq
  \expec^{\prob^{v,x}_T}\bra{\int_0^T \frac{\underline{\kappa}^2}{2} (\nabla
    h)' A \nabla h (T-s, X_s) \, ds + \underline{\kappa}\int_0^T (\nabla h)' a
    dW^T_s}\\
    & \geq -\log \expec^{\prob^{v, x}_T}\bra{\mathcal{E}\pare{-\uk \int (\nabla h)' a dW^T_s}_T}\geq 0,
  \end{split}
 \]
 where $W^T$ is a $\prob^{v,x}_T-$Brownian motion and the second inequality follows from Jensen's inequality.  Thus, since $h(0,x) =
 \phi_0(x)-\hat{v}(x)$, for any $T \geq 0$ and $x\in E$,
\begin{equation*}
h(T,x) \leq \expec^{\prob^{v,x}_T}\bra{\phi_0(X_T) - \hat{v}(X_T)} \leq C +
(K+1)\expec^{\prob^{v,x}_T}\bra{\phi_0(X_T)} \leq C + (K+1)C(1+\phi_0(x)+\hat{v}_{-}(x)),
\end{equation*}
where the second inequality uses the first inequality in \eqref{eq: main ub main ub} and \eqref{eq: hat v lb}, the third inequality uses \eqref{eq: pTx phi0 expec
  ub}. Hence \eqref{eq: h ub} now holds by taking $J$ large
enough, finishing the proof.
\end{proof}

For $0\leq t\leq T$ and $x\in E$ set
\begin{equation}\label{eq: fT def}
f^{t,T}(x) := \frac{1}{2}\hespalt{x}{}{\int_0^t (\nabla h)' \overline{A} \nabla h (T-s, X_s)\, ds}.
\end{equation}
The next result gives a weak form of the convergence in Theorem \ref{thm: conv} $i)$.

\begin{prop}\label{prop: quad var conv}
 For all $t\geq 0$,
 \begin{equation}\label{eq: int quad var conv}
 \lim_{T\rightarrow \infty} \int_{E} f^{t,T}(x)\hat{m}(x) \, dx =0.
 \end{equation}
\end{prop}

\begin{proof}

 Corollary \ref{cor: ui} and Ito's formula applied to $h(T-\cdot, X_\cdot)$ imply that
\begin{equation}\label{eq: f_h_compare}
f^{t,T}(x) = h(T,x) - \hespalt{x}{}{h(T-t,X_t)}.
\end{equation}
Let $\hat{p}(t, x, y)$ denote the transition density of $X$ under
$\hat{\prob}^x$. Recall from \cite[pp. 179]{Pinsky} that
\begin{equation}\label{eq: hat_m_invariance}
\hat{m}(y) = \int_{E} \hat{p}(t,
x, y) \hat{m}(x) dx, \quad  \text{for any } t>0 \text{ and } y\in E.
\end{equation}
Thus
\begin{equation}\label{eq: t1}
\begin{split}
\int_E f^{t,T}(x)\hat{m}(x)dx &= \int_Eh(T,x)\hat{m}(x)dx - \iint_E
\hat{p}(t,x,y)h(T-t,y)\hat{m}(x)dydx\\
&=\int_E h(T,x)\hat{m}(x)dx - \int_Eh(T-t,y)\hat{m}(y)dy.
\end{split}
\end{equation}
Set $l(T) := \int_E h(T,x)\hat{m}(x)dx$.
It then follows from \eqref{eq: t1} and $f^{t,T}(x)\geq 0$ that $l(T)\geq l(T-t)$ for all $0\leq t\leq T$.  Therefore $l(T)$ is
increasing in $T$ and hence $\lim_{T\rightarrow\infty}l(T)$ exists.
Furthermore, by \eqref{eq: h ub} we know that
$
\sup_{T\geq 0}l(T)\leq \int_E \cJ(x)\hat{m}(x)dx < \infty,
$
hence $\lim_{T\rightarrow\infty}l(T) = l < \infty$.  Sending $T\rightarrow \infty$ on both sides of \eqref{eq: t1}, we have
\begin{equation*}
\lim_{T\rightarrow\infty}\int_E f^{t,T}(x)\hat{m}(x)dx =
\lim_{T\rightarrow\infty}\left(l(T)-l(T-t)\right) = l-l = 0.
\end{equation*}
\end{proof}

In order to remove the integral with respect to the invariant density in \eqref{eq: int quad var conv}, we need the following result.

\begin{lem}\label{lem: bdd equicont}
For any fixed $t> 0$ and $n\in\Natural$, the family of functions on $E$ given by $\cbra{f^{t,T}(\cdot); T\geq t}$ is uniformly bounded and equicontinuous on $E_n$.
\end{lem}

\begin{proof}
 Define
$k^{t,T}(s,x) :=\hespalt{x}{}{h(T-t, X_s)}$, for $s\leq t\leq T$ and $x\in E$,
so that \eqref{eq: f_h_compare} becomes $f^{t,T}(x) = h(T,x) - k^{t,T}(t,x)$. We will prove, for any $E_n \subset E_m$ and $t>0$,
\begin{enumerate}[a)]
\item $\{k^{t,T}(s,\cdot); T\geq t, t\geq s\geq t/2\}$ is uniformly bounded on $E_m$.
\item $\{h(T, \cdot); T\geq t/2\}$ is
uniformly bounded on $E_m$.
\item both $\{k^{t,T}(t, \cdot); T\geq t\}$ and $\{h(T, \cdot); T\geq t\}$ are equicontinuous in $E_n$.
\end{enumerate}

Let us first handle $k^{t,T}$. We have from \eqref{eq: h lb} and \eqref{eq: h ub} that
\begin{equation*}
-C -2\hespalt{x}{}{\phi_0(X_{T-t+s})} = -C -2 \hat{\expec}^{x} \bra{\hat{\expec}^{X_s}[\phi_0(X_{T-t})]} \leq k^{t,T}(s,x) \leq \hespalt{x}{}{\cJ(X_s)},
\end{equation*}
for $T-t\geq 0$ and $t\geq s\geq t/2$.
Since $\phi_0, \cJ\in \bL^1(E,\hat{m})$ from Lemmas \ref{lem: phi_0 int} and \ref{lem: main ub}, it then follows from Remark \ref{rem: uni-local bdd} $ii)$ that both $\sup_{s\geq t/2} \hat{\expec}^x[\cJ(X_s)]$ and $\sup_{T\geq t, s\geq t/2}\hat{\expec}^x[\phi_0(X_{T-t+s})]$ are bounded in $E_m$. Therefore, assertion a) is verified. Similarly,
\eqref{eq: h lb} and  \eqref{eq: h ub} imply that
\begin{equation*}
-C - 2\hespalt{x}{}{\phi_0(X_T)} \leq h(T,x) \leq \cJ(x), \quad \text{for} \quad T\geq t/2.
\end{equation*}
Then again, assertion b) follows from $\phi_0\in \bL^1(E,\hat{m})$ and Remark \ref{rem: uni-local bdd} $ii)$.

To prove $\{k^{t,T}(t, \cdot); T\geq t\}$ is equicontinuous in $E_n$,
 one can show that $k^{t,T}\in C^{1,2}((0,t)\times E_n) \cap C([0,t]\times \overline{E_n})$ and satisfies
 \begin{align*}
  \partial_s k = \cL^{\hat{v}}k \quad \text{ in } (0, t) \times E_n.
 \end{align*}
 This result essentially follows from \cite{Heath-Schweizer}, and its proof is carried out in \cite[Lemma A.3]{Guasoni-Robertson-fundsep}.
 It then follows from the interior Schauder estimates (cf. \cite[Theorem
 2.15]{friedman-parabolic}) that, for any $E_n\subset E_m$ with $n< m$,
 $\max_{\overline{E_n}} |\nabla k^{t,T}(t, \cdot)|$ is bounded from above by a
 constant which only depend on the dimension of the problem,
 $\max_{[t/2, t]\times \overline{E_m}}|k^{t,T}|$, maximum and minimum of eigenvalues
 of $A$ in $\overline{E_m}$, the distance from the boundary of
 $E_n$ to the boundary of $E_m$, and finally $t$. In particular, the uniform bounds in a) implies that this
 upper bound on $\max_{\overline{E_n}} |\nabla k^{t,T}(t, \cdot)|$  is independent of $T$. Therefore $\{k^{t,T}(t, \cdot); T\geq t\}$ is equicontinuous in $E_n$.

Now, $h$ satisfies \eqref{eq: pde h} for all $T > 0$ and $x\in E_m$. Moreover, we have seen from b) that  $\{h(T,
 \cdot); T\geq t/2\}$ is uniformly bounded in $E_m$. It then follows from
 \cite[Theorem V.3.1]{Lad} that, for any $E_n\subset E_m$ with $n<m$ and $T\geq t$,
 $\max_{\overline{E_n}}|\nabla h(T,\cdot)|$ is bounded by a constant which
 only depends on the dimension of the problem, uniform bounds for $h$ in b), the minimum
 and maximum eigenvalue of $A$ in $\overline{E_m}$, distance from boundary of $E_n$ to boundary of $E_m$, and finally $t$. Therefore, $\{h(T, \cdot); T\geq t\}$ is equicontinuous in $E_n$ as
 well.
\end{proof}

\begin{rem}\label{rem: h equicont}
 For later development, we record from the previous proof that $\{h(T, \cdot); T\geq t\}$ is uniformly bounded and equicontinuous on $E_n$ for any $n$.
\end{rem}

With these preparations we are able to prove Theorems \ref{thm: conv} and
\ref{thm: pointwise conv}.

\begin{proof}[Proof of Theorem \ref{thm: conv}] Suppose that the convergence
  in $i)$ does not hold, then there exist $\epsilon>0$, $E_n$, and a sequence
  $(T_i)_{i}$ such that $\sup_{E_n}f^{t,T_i}(x)\geq \epsilon$ for all $i$. Owing to
  Lemma \ref{lem: bdd equicont}, the Arzela-Ascoli theorem implies, taking a
  subsequence if necessary, $f^{t,T_i}$ converge to some continuous function
  $\hat{f}$ uniformly in $E_n$. Note that $\sup_{E_n}|f^{t, T_i}- \hat{f}| + \sup_{E_n} \hat{f} \geq \sup_{E_n} f^{t, T_i}$. Sending $T_i \rightarrow \infty$, the uniform convergence and the choice of $f^{t, T_i}$ implies $\sup_{E_n} \hat{f}(x) \geq \epsilon$. Since $\hat{f}$ is continuous, there exists a subdomain of $D\subset E_n$ such that $\hat{f}\geq \epsilon/2$ on $D$. However, this contradicts with Proposition \ref{prop: quad var
    conv} when the bounded convergence theorem is applied to the family of functions
  $(f^{t, T_i}\indic_{D})_{i\in \Natural}$.

 To prove the statement $ii)$, utilizing \eqref{eq: pde h} and applying Ito's formula to $h(T-\cdot, X_\cdot)$, we obtain
 \[
  \sup_{0\leq u\leq t} |h(T, x) - h(T-u, X_u)| \leq \frac12 \int_0^t (\nabla h)' \overline{A} \nabla h(T-s, X_s)\, ds + \sup_{0\leq u\leq t} \left|\int_0^u (\nabla h)' \,a\, d\hat{W}_s\right|.
 \]
 Taking the $\hat{\prob}^x$-expectation on both sides and using Burkholder-Davis-Gundy inequality, we obtain
  \begin{align*}
   & \sup_{E_n}\wh{\expec}^x \bra{\sup_{0\leq u\leq t} |h(T, x) - h(T-u, X_u)|} \\
   & \leq \frac12\,\sup_{E_n}\wh{\expec}^x \bra{\int_0^t(\nabla h)' \overline{A} \nabla h(T-s, X_s)\, ds} + c\pare{\sup_{E_n}\wh{\expec}^x \bra{\int_0^t(\nabla h)' A \nabla h(T-s, X_s)\, ds}}^{\frac12}\\
   &\rightarrow 0, \quad \text{ as } T\rightarrow \infty, \text{ for any } E_n,
 \end{align*}
 where the last convergence follows from $i)$ and $A\leq \overline{A}/\underline{\kappa}$.
\end{proof}

\begin{proof}[Proof of Theorem \ref{thm: pointwise conv}]

The proof of Theorem \ref{thm: pointwise conv} follows a similar argument as
those presented in the proofs of \cite[Proposition 4.3 and Theorems 1.3,
1.4, 4.1]{Ichihara-Sheu}, and hence only connections to those proofs are given.
Regarding \cite[Proposition 4.3]{Ichihara-Sheu},  for constants $S,T>0$, using
\eqref{eq: h lb exp ub} at $t = S + T$ and $\tau = S$ it follows that
\begin{equation*}
\hespalt{x}{}{e^{\uk h(T,X_S)}} \leq e^{\uk h(T+S,x)}.
\end{equation*}
Furthermore, \eqref{eq: v0 bound} implies, for any fixed $T >0$ that
$h(T,x)\leq \phi_0(x) -\hat{v}(x) + C_T$ for some $C_T > 0$. This, combined
with Proposition \ref{prop: e-existence} $iii)$ (with $\phi = \phi_0$) imply
that $e^{\uk h(T,\cdot)}\in \mathbb{L}^1(E,\hat{m})$ for any fixed $T>0$ and hence
Remark \ref{rem: uni-local bdd} $iii)$ implies
\begin{equation*}
\lim_{S\uparrow\infty}\hespalt{x}{}{e^{\uk h(T,X_S)}} = \int_E e^{\uk
  h(T,y)}\hat{m}(y)dy.
\end{equation*}
Thus, the conclusions of \cite[Proposition 4.3]{Ichihara-Sheu} follow by
repeating their proof, noting that the role of $-k_1w(t,x)$ therein is now
played by $\uk h(t,x)$ here.  Now, $i)$ in Theorem \ref{thm: pointwise conv}
follows by repeating the argument of \cite[Theorem 4.1]{Ichihara-Sheu} and using Remark \ref{rem: h equicont}.

As for part $ii)$ in Theorem \ref{thm: pointwise conv}, we essentially repeat
the steps within the proof of \cite[Theorem 1.4]{Ichihara-Sheu}. Namely, using
interior estimates for quasi-linear parabolic equations in \cite[Theorem V.3.1]{Lad} and Remark
\ref{rem: h equicont} it follows that there are constants $C_n>0$ and $\gamma \in (0,1)$ such that
for $x,y\in E_n$ and $s,\tilde{s} > t $: $|\nabla h(s,y) - \nabla
h(\tilde{s},x)| \leq C_n|s-\tilde{s}|^\gamma$.  Now, define
 \begin{equation}\label{eq: f conv}
  f(n,T):= \int_{E_n} (\nabla h)' \barA \nabla h(T, y) \hat{m}(y)\, dy.
 \end{equation}
It thus follows that $f(n,T)$ is uniformly continuous in $(t,\infty)$.  Next we claim that $\lim_{T\uparrow\infty} f(n,T) = 0$ for any $n$. Indeed, recall from Proposition \ref{prop: quad var conv} that
 \[
 0 =\lim_{T\rightarrow \infty} \int_{E} \hat{\expec}^x\bra{\int_0^t (\nabla h)'
   \overline{A} \nabla h(T-s, X_s) \, ds} \hat{m}(x) \, dx=0,\quad \textrm{for
   any } t>0.
 \]
 Applying Fubini's theorem and \eqref{eq: hat_m_invariance} to the previous convergence yields
\begin{equation}\label{eq: f int conv}
\lim_{T\rightarrow\infty}\int_0^t f(n,T-s)ds =0.
\end{equation}
Therefore, as shown in the proof of \cite[Theorem 1.4]{Ichihara-Sheu}, that
$f(n,T)\rightarrow 0$ follows by the uniform continuity of $f(n,\cdot)$. The
remaining steps of the proof are identical to those in \cite[Theorem 1.4]{Ichihara-Sheu}.
\end{proof}

\section{Proofs from Section \ref{sec: conv Rd sd}}\label{sec: sd proof}

\subsection{Proof of Theorem \ref{thm: conv Rd}}\label{subsec: Rd proof}

Theorem \ref{thm: pointwise conv} has been proved in \cite{Ichihara-Sheu} when $E=\Real^d$ and $A=I_d$. When $A$ and $\overline{A}$ are local elliptic satisfying Assumption \ref{ass: coeff}, the same calculation as in \cite[Proposition 5.1]{Ichihara-Sheu} shows, when Assumption \ref{ass: Rd growth} holds, there exist $\epsilon_0, C>0$ and $0<\underline{c} <\overline{c}$ such that $\phi_0 = (c/2) |x|^2$ for any $c\in (\underline{c}, \overline{c})$ satisfies
\[
 \fF[\phi_0](x) = \frac{1}{2}c\tr(A)+\frac{1}{2}c^2x'\bar{A}x + cx'B + V \leq C + \left(\frac{\ok}{2}\alpha_1 c^2 - \beta_1 c - \gamma_1\right)|x|^2 \leq C - \epsilon_0 |x|^2, \quad x\in\Real^d.
\]
Indeed, for $\gamma_1 > 0$ one can take $0\leq \underline{c} < \overline{c}$ for $\overline{c}$ sufficiently small, while for $\gamma_1< 0, \beta_1 > 0$ one can use part $iv-b)$ of Assumption \ref{ass: Rd growth} to find $0 < \underline{c} < \overline{c}$.  Therefore, Assumption \ref{ass: long_run} is satisfied and Assumption \ref{ass: long_run strong} holds when $\delta>1$ satisfies $c \delta <\overline{c}$. On the other hand, Assumption \ref{ass: Rd growth} i) implies that $A$ is bounded and $B + \overline{A} \nabla \phi_0$ has at most linear growth. Thus the coordinate process does not explode $\prob^{\phi_0, x}$-a.s. for any $x\in \Real^d$, implying that Assumption \ref{ass: wellpose Lphi0} holds. As for Assumption \ref{ass: psi}, take $\psi_0(x) =-(\tilde{c}/2) |x|^2$ for $\tilde{c}>0$, the second convergence in \eqref{eq: psi_0 hat v} and the first inequality in \eqref{eq: main ub main ub} clearly hold. For the second inequality in \eqref{eq: main ub main ub},
\[
 \fF[\delta \phi_0] + \alpha (\delta \phi_0 - \psi_0) \leq C - (\epsilon_0 - (\alpha/2) (\delta c + \tilde{c}))|x|^2,
\]
which is bounded from above by $C$ when $\alpha$ is sufficiently small. Finally, it remains to find $\tilde{c}$ such that the first convergence in \eqref{eq: psi_0 hat v} is verified. To this end,
\begin{align*}
  \fF[\psi_0](x) &= -\frac12 \tilde{c} \tr(A) + \frac12 \tilde{c}^2 x' \barA x - \tilde{c} x' B + V\geq C + \frac{\uk}{2}\, \tilde{c}^2 x' A x+ (\tilde{c} \beta_1 - \gamma_2) |x|^2,
 \end{align*}
 where the inequality is a result of $\barA \geq \uk A$ and Assumption \ref{ass: Rd growth} i)-iii). When $\beta_1 >0$, choose $\tilde{c}$ sufficiently large such that $\tilde{c}\beta_1 > \gamma_2$. When $\gamma_1>0$ and $\beta_1\leq 0$, Assumption \ref{ass: Rd growth} iv-a) yields $(\uk/2) \tilde{c}^2 x' A x \geq (\uk/2) \tilde{c}^2 (\alpha_2 |x|^2 - C_3)$. Thus choose $\tilde{c}$ sufficiently large such that $(\uk/2) \tilde{c}^2 \alpha_2 + \tilde{c}\beta_1 - \gamma_2 >0$. In conclusion, all assumptions of Theorem \ref{thm: conv} are satisfied, hence statements of Theorems \ref{thm: conv} and \ref{thm: pointwise conv} follow.

\subsection{Proof of Theorem \ref{thm: conv sd}}\label{subsec: sd proof}

\subsubsection{Preliminaries}\label{subsubsec: prelim} The assumptions of Theorems
\ref{thm: conv} and \ref{thm: pointwise conv} are now verified via Assumptions
\ref{ass: coeff sd} -- \ref{ass: H_eps}, which are enforced throughout. To
ease notation, the argument $x$ is suppressed when writing any function
$f(x)$; for example,  $\tr(f(x)x g(x)x)$ will be written as $\tr(fxgx)$.  The
following basic identities and inequalities are used repeatedly.  The first
one concerns derivatives of the functions $\log(\det(x))$ and $\norm{x}$
respectively, and holds for $i,j,k,l =1 ,...,d$:
\begin{equation}\label{eq: log_norm_derivs}
\begin{split}
D_{(ij)}\log(\det(x)) &= x^{-1}_{ij},\qquad D^2_{(ij),(kl)}\log(\det(x)) = -(x^{-1})_{il}(x^{-1})_{jk},\\
D_{(ij)}\norm{x} &= \frac{x_{ij}}{\norm{x}},\qquad D^2_{(ij),(kl)} \norm{x} = \frac{\delta_{(ij),(kl)}}{\norm{x}} - \frac{x_{ij}x_{kl}}{\norm{x}^3},
\end{split}
\end{equation}
where $\delta_{(ij),(kl)} = 1$ if $i=k,j=l$ and $0$ otherwise. Next, we give
an identity, which follows from the discussion below \eqref{eq: f_g_def}:
\begin{equation}\label{eq: quad_form_matrix}
\begin{split}
\sum_{i,j,k,l=1}^d \theta_{ij}\tr\left(a^{ij}(a^{kl})'\right) \psi_{kl} &=
4\tr(f\psi g\theta);\qquad \theta,\psi\in\sd,
\end{split}
\end{equation}
Now, \eqref{eq: log_norm_derivs}, along with the definitions of $\cL$ and $H_\delta$
from \eqref{eq: L_phi_matrix def} and \eqref{eq: Hdelta def} respectively, give
\begin{equation}\label{eq: log_norm_ops}
\cL(\log(\det(x))) = H_0;\qquad \cL(\norm{x}) = \frac{1}{\norm{x}}\left(\tr(f'g) + \tr(f)\tr(g) - \frac{2}{\norm{x}^2}\tr(fxgx) + \tr(B'x)\right).
\end{equation}
On the other hand, for $\theta,\psi,\eta\in \sd$:
\begin{equation}\label{eq: pos_def_trace}
\tr(\theta\psi)\leq \tr(\theta)\tr(\psi);\qquad \tr(\theta\psi\eta\psi)\leq \tr(\theta)\tr(\eta)\norm{\psi}^2.
\end{equation}
Note that the first inequality in \eqref{eq: pos_def_trace} also holds for  $\theta\in\sd$ and $\psi\in\mathbb{M}^d$ with $\psi+\psi' \in\sd$. This is because $\tr(\theta\psi)= (1/2) \tr(\theta(\psi+\psi')) \leq (1/2) \tr(\theta)\tr(\psi+\psi')=\tr(\theta)\tr(\psi)$. Lastly for any constants $a,b>0$,
\begin{equation}\label{eq: lim_norm_det}
\lim_{\theta \rightarrow \partial \sd} -a\log(\det(\theta)) + b\norm{\theta} = \infty.
\end{equation}
This convergence is clear when $\det(\theta)\downarrow 0$. When $\norm{\theta}\uparrow \infty$,
since $\det(\theta) = \prod_{i=1}^d \lambda_i$ and $\norm{\theta}=\sqrt{\sum_{i=1}^d \lambda_i^2}$, where $(\lambda_i)_{i=1\dots d}>0$ are the eigenvalues of $\theta$, counting multiplicity, then
\eqref{eq: lim_norm_det} follows from Jensen and H\"{o}lder's inequalities.

\subsubsection{Proofs} Let us first show  $\fF$ in \eqref{eq: F sd} is locally elliptic.

\begin{lem}\label{lem: elliptic sd}
 Let Assumption \ref{ass: coeff sd} $i)$ hold. Then for each $E_n \subset \sd$, there exists $c_n>0$ such that
  \[\sum_{i,j,k,l=1}^d \theta_{ij} \tr(a^{ij}(a^{kl})')(x) \theta_{kl} =
  4\tr(f\theta g\theta)(x)\geq c_n \|\theta\|^2, \quad  \mbox{for any } x\in
  E_n, \theta\in \sdall. \]
\end{lem}

\begin{proof}
Applying \eqref{eq: quad_form_matrix} for $\theta = \psi\in \sd$ gives $\sum_{i,j,k,l=1}^d \theta_{ij} \tr\pare{a^{ij}(a^{kl})'} \theta_{kl} = 4\, \tr(f \theta g \theta)$. Now, $\tr(f \theta g \theta) = (\mbox{vec} \theta)' (f\otimes g) (\mbox{vec}
\theta)$, cf. \cite[Chapter 4, Problem 25]{Horn-Johnson}, where
$\mbox{vec}(\theta) \in\Real^{d^2}$ is obtained by stacking columns of
$\theta$ on top of one another. From \cite[Corollary 4.2.13]{Horn-Johnson} it
follows that $f\otimes g$ is positive definite if both $f$ and $g$ are positive definite. Hence Assumption \ref{ass: coeff sd} $i)$ ensures the existence of $c_n>0$ such that $(\mbox{vec} \theta)' (f\otimes g) (\mbox{vec} \theta)\geq c_n |\mbox{vec} \theta|^2 = c_n \|\theta\|^2$ on $E_n$, proving the result.
\end{proof}

Let us now study the Lyapunov function $\phi_0$.
Recall $\phi_0$ and the cutoff function $\eta$ from \eqref{eq: phi0 matrix def} and, for given $\overline{c},
\underline{c} > 0$ set $\phi_0^{(1)}(x)\dfn -\underline{c}\log(\det(x))$ and
$\phi_0^{(2)}(x) = \overline{c}\norm{x}\eta(\norm{x})$ so that $\phi_0 =
\phi^{(1)}_0 + \phi^{(2)}_0 + C$. We first derive an upper bound for $\fF[\phi_0]$.

\begin{lem}\label{lem: Fphi_0 ub}
There exists a constant $C$, depending on $\oc$ but not on $\uc$, such that
\begin{equation}\label{eq: fF big ub 2}
\fF[\phi_0](x) \leq -\underline{c}H_{4\ok\underline{c}}(x)
  -\left(\gamma_1 +
    \beta_1\overline{c}-4\ok\alpha_1\overline{c}^2\right) \norm{x} \,\indic_{\{\norm{x} >
  n_0 + 2\}} + C, \quad \text{ for } x\in \sd.
\end{equation}
\end{lem}

\begin{proof}
By the definition of $\mathfrak{F}$ and Assumption \ref{ass: coeff sd} $ii)$:
  \begin{equation}\label{eq: fF big ub}
\begin{split}
  \fF[\phi_0] &\leq  \cL \phi^{(1)}_0 + \cL \phi^{(2)}_0 + \frac{\ok}{2}
  \sum_{i,j,k,l=1}^d (D_{ij}\phi^{(1)}_0 +
  D_{(ij)}\phi^{(2)}_0)\tr\left(a^{ij}(a^{kl})'\right)(D_{(kl)}\phi^{(1)}_0 +
  D_{(kl)}\phi^{(2)}_0) + V,\\
&\leq  \cL \phi_0^{(1)} + \ok\sum_{i,j,k,l=1}^dD_{(ij)}\phi_0^{(1)} \tr\pare{a^{ij}(a^{kl})'}
  D_{(kl)} \phi_0^{(1)}\\
&\qquad +\cL\phi^{(2)}_0  + \ok
\sum_{i,j,k,l=1}^d D_{(ij)}\phi_0^{(2)} \tr\pare{a^{ij}(a^{kl})'}
  D_{(kl)} \phi_0^{(2)} + V.\\
 \end{split}
\end{equation}
In what follows, each term on the right-hand-side will be estimated. First,
\eqref{eq: quad_form_matrix}, \eqref{eq: log_norm_ops}, and the definition of $H_\delta$ in \eqref{eq: Hdelta def} yield:
\begin{equation}\label{eq: phi10 big ub}
\begin{split}
\cL \phi_0^{(1)}  + \ok \sum_{i,j,k,l=1}^d D_{(ij)}\phi_0^{(1)}
\tr\pare{a^{ij}(a^{kl})'} D_{(kl)} \phi_0^{(1)}&= -\underline{c}H_{4\ok\underline{c}}.
\end{split}
\end{equation}
As for $\phi^{(2)}_0$, when $\norm{x} >n_0+2$, $\phi^{(2)}_0(x) = \norm{x}$
and hence by \eqref{eq: quad_form_matrix} and  \eqref{eq: log_norm_ops}:
\begin{equation}\label{eq: Lphi2 ub x}
\begin{split}
\cL \phi_0^{(2)}& + \ok \sum_{i,j,k,l=1}^d D_{(ij)}\phi_0^{(2)}
\tr\pare{a^{ij}(a^{kl})'} D_{(kl)} \phi_0^{(2)}\\
&=
\frac{\overline{c}}{\norm{x}}\left(\tr(f'g) + \tr(f)\tr(g) +
  \left(\frac{4\overline{c}\ok}{\norm{x}}-\frac{2}{\norm{x}^2}\right)\tr(fxgx)
  + \tr(B'x)\right).
\end{split}
\end{equation}
Assumption \ref{ass: growth sd} is now used to refine the right-hand-side.  Since
$\tr(f'g) \leq \tr(f)\tr(g)$ for $f,g\in \sd$ (cf. \eqref{eq: pos_def_trace}), Assumption \ref{ass: growth sd} $i)$  yields
\begin{equation*}
\frac{\overline{c}}{\norm{x}}\left(\tr(f'g) + \tr(f)\tr(g)\right) \leq
2\alpha_1\overline{c}, \quad \norm{x} >n_0+2.
\end{equation*}
Lemma \ref{lem: elliptic sd} implies $\tr(fxgx) \geq 0$. This, and $\tr(fxgx) \leq \tr(f) \tr(g) \norm{x}^2$ (cf. \eqref{eq: pos_def_trace} again) gives, in light of Assumption \ref{ass: growth sd} $i)$, that
\begin{equation*}
\frac{\overline{c}}{\norm{x}}\left(\frac{4\overline{c}\ok}{\norm{x}}-\frac{2}{\norm{x}^2}\right)\tr(fxgx)
\leq 4\overline{c}^2\ok\alpha_1\norm{x}, \quad \norm{x} > n_0+2.
\end{equation*}
Lastly, Assumption \ref{ass: growth sd} $ii)$ gives
\begin{equation*}
\frac{\overline{c}}{\norm{x}}\tr(B'x) \leq -\beta_1\overline{c}\norm{x} +
\frac{C_1\overline{c}}{n_0+2}, \quad \norm{x}>n_0+2.
\end{equation*}
Putting previous three estimates back to \eqref{eq: Lphi2 ub x} yields
\begin{equation}\label{eq: phi20 big ub}
\begin{split}
\cL \phi_0^{(2)} &+ \ok \sum_{i,j,k,l=1}^d D_{(ij)}\phi_0^{(2)}
\tr\pare{a^{ij}(a^{kl})'} D_{(kl)} \phi_0^{(2)}
\leq  C -\left(\beta_1\overline{c}-4\ok\alpha_1\overline{c}^2\right)\norm{x},
\end{split}
\end{equation}
when $\norm{x}>n_0+2$. Here $C$ is a constant which depends linearly on $\overline{c}$.

On the other
hand, when $\norm{x}\leq n_0 + 2$, since $\phi^{(2)}_0$ and its
derivatives are bounded for bounded $\norm{x}$, one can show the left-hand-side of \eqref{eq: Lphi2 ub x} is bounded from above by a constant. Combining previous estimates on different parts of $\sd$ yields
\begin{equation}\label{eq: Lphi2 ub}
\cL \phi_0^{(2)} + \ok \sum_{i,j,k,l=1}^d D_{(ij)}\phi_0^{(2)}
\tr\pare{a^{ij}(a^{kl})'} D_{(kl)} \phi_0^{(2)}
\leq  C -\left(\beta_1\overline{c}-4\ok\alpha_1\overline{c}^2\right)\norm{x} \,\indic_{\{\norm{x} >n_0+2\}}.
\end{equation}

Now putting \eqref{eq: phi10 big ub} and \eqref{eq: Lphi2 ub} back into
\eqref{eq: fF big ub}, and utilizing the upper bound of $V$ in
Assumption \ref{ass: growth sd} $iii)$, we confirm \eqref{eq: fF big ub 2}.
\end{proof}

The upper bound in \eqref{eq: fF big ub 2} is then used to identify the Lyapunov function and verify its properties.

\begin{lem}\label{lem: quad bounds sd}
For the $\epsilon$ of Assumption \ref{ass: H_eps}, there exist $C>0$ and $0 <
\overline{c}_l < \overline{c}_h$ such that for any $0 < \underline{c} <
\epsilon/(4\ok)$ and $\overline{c}_l < \overline{c} < \overline{c}_h$ the function $\phi_0$ in \eqref{eq: phi0 matrix def}
is nonnegative on $\sd$ and
satisfies $\lim_{n\uparrow \infty} \sup_{x\in \sd \setminus
  E_n}\fF[\phi_0](x)=-\infty$. Therefore, Assumption \ref{ass: long_run} holds.
\end{lem}

\begin{proof}
Since $\tr(fxgx) > 0$ for $x\in \sd$, $H_{\delta}$ is decreasing in $\delta$. Hence
for $\underline{c} < \epsilon / (4\ok)$, \eqref{eq: fF big ub 2} gives
\begin{equation*}
\fF[\phi_0] \leq -\underline{c}H_{\epsilon}(x) -\left(\gamma_1 +
    \beta_1\overline{c}-4\ok\alpha_1\overline{c}^2\right)\norm{x} \,\indic_{\{\norm{x} >
  n_0 + 2\}} + C.
\end{equation*}
Assume for now that there exist $\epsilon_0 > 0$ and $0 < \overline{c}_l <
\overline{c}_h$ such that
\begin{equation}\label{eq: neg c range}
\gamma_1 + \beta_1\overline{c}-4\ok\alpha_1\overline{c}^2 \geq \epsilon_0, \quad \text{ for any } \overline{c}\in (\overline{c}_l, \overline{c}_h).
\end{equation}
For such $\underline{c}$ and $\overline{c}$, the previous two inequalities combined imply
\begin{equation}\label{eq: new fF big ub}
\fF[\phi_0] \leq -\underline{c}H_{\epsilon}(x)  -\epsilon_0\norm{x} \,\indic_{\{\norm{x} >
  n_0 + 2\}} + C.
\end{equation}
By Assumption \ref{ass: H_eps} $i)$,
$
\fF[\phi_0] \leq C  -\epsilon_0\norm{x} \,\indic_{\{\norm{x} >
  n_0 + 2\}} \rightarrow -\infty
$, as $\norm{x}\uparrow \infty$. Moreover, Assumption \ref{ass: H_eps} $ii)$ implies
$\lim_{\det(x)\downarrow 0} H_\epsilon(x) =\infty$ and thus $\fF[\phi_0] \leq
C- \uc H_\epsilon(x)\rightarrow -\infty$ as $\det(x)\downarrow 0$. Combining these two cases, the assertion $\lim_{n\uparrow \infty} \sup_{x\in \sd \setminus E_n}\fF[\phi_0](x)=-\infty$ is confirmed.

To show \eqref{eq: neg c range}, we use Assumption \ref{ass: growth sd} $iv)$.  When $\gamma_1 > 0$ one can
take $\epsilon_0 = \gamma_1 /3$ and $\overline{c}_l = \overline{c}_h / 2$ for
some small enough $\overline{c}_h > 0$.  When $\gamma_1 \leq 0$ and $\beta_1 >
0$, $\beta_1^2 + 16 \ok \alpha_1 \gamma_1 >0$ holds due to Assumption \ref{ass: growth sd} $iv-b)$. Then there exists some sufficiently small $\epsilon_0$ such that  $\beta_1^2 - 16 \ok \alpha_1 (-\gamma_1 + \epsilon_0)>0$. Hence one can take any $\oc_l < \oc_h$ satisfying $\overline{c}^- < \overline{c}_l < \overline{c}_h < \overline{c}^+$, where $\overline{c}^\pm>0$ are two roots of $-4\ok \alpha_1 \oc^2 + \beta_1 \oc + \gamma_1 - \epsilon_0=0$.

Finally, it follows from \eqref{eq: lim_norm_det} that $\phi_0$ can be made nonnegative by adding a sufficiently large constant $C$ to $\phi_0^{(1)}+\phi_0^{(2)}$.

\nada{
To show the previous convergence, first observe that
$\phi_0^{(1)}+\phi_0^{(2)}\rightarrow\infty$ as $\det(x)\downarrow 0$ irrespective of
$\norm{x}$. On the other hand, let
$\lambda_i>0$, $i=1,\cdots, d$,  be eigenvalues of $x$. Then $\det(x) = \prod_{i=1}^{d} \lambda_i$ and $\norm{x}^2 = \sum_{i=1}^d
\lambda^2_i$. By Jensen and H\"{o}lder's inequality,
\begin{equation*}\label{eq: logdet to norm}
\begin{split}
\log\det(x) = d\left(\frac{1}{d}\sum_{i=1}^d \log
  \lambda_i\right) \leq d\log\left(\frac{1}{d}\sum_{i=1}^{d}\lambda_i\right)
\leq d\log\left(\frac{1}{\sqrt{d}}\sqrt{\sum_{i=1}^{d}\lambda^2_i}\right) =
d\log\left(\frac{\norm{x}}{\sqrt{d}}\right).
\end{split}
\end{equation*}
Thus, $\phi_0^{(1)} + \phi_0^{(2)} \geq -\uc
d\log(\norm{x}/\sqrt{d}) + \oc\norm{x}\eta(\norm{x})\rightarrow\infty$ as
$\norm{x}\rightarrow\infty$.}
\end{proof}

\begin{cor}\label{cor: phi prop sd}
 The following statements hold.
 \begin{enumerate}[i)]
\item When $\uc < \epsilon/(8\uk)$, the martingale problem for
    $\cL^{\phi_0}$ is well-posed on $\sd$. Hence, Assumption \ref{ass:
      wellpose Lphi0} is satisfied.
  \item There exists $\delta>1$ such that $\lim_{n\uparrow \infty}\sup_{x\in
      \sd \setminus E_n} \fF[\delta \phi_0](x)=-\infty$. Hence Assumption
    \ref{ass: long_run strong} is satisfied.
 \end{enumerate}
\end{cor}

\begin{proof}
Part $ii)$ follows from \eqref{eq: phi0 matrix def} and from Lemma \ref{lem: quad
  bounds sd} by taking $\delta>1$ such that $\uc\delta < \epsilon/(4\uk)$ and
$\oc_l <\oc\delta < \overline{c}_h$.

To prove part $i)$, note that $\cL^{\phi_0}\phi_0 = \cL \phi_0 + \sum_{i,j,k,l=1}^{d}D_{(ij)}\phi_0
\bar{A}_{(ij),(kl)} D_{(kl)}\phi_0$, an upper bound for which is obtained by following \eqref{eq: fF big ub}, replacing $\uk$ by $2\uk$ and disregarding $V$. Then the same estimates leading to \eqref{eq: phi10 big ub} and \eqref{eq: Lphi2 ub} yield
\begin{equation}\label{eq: Lphi0phi0}
 \cL^{\phi_0}(\phi_0)(x) \leq -\uc H_{8\ok \uc}(x) - (\beta_1 \oc - 8\ok \alpha_1 \oc^2) \norm{x} \, \indic_{\{\norm{x}>n_0+2\}} + C.
\end{equation}
From \eqref{eq: lim_norm_det}, $\phi_0(x) \uparrow \infty$ as either $\det(x)
\downarrow 0$ or $\norm{x} \uparrow \infty$. If we can find $\lambda>0$ such
that $(\cL^{\phi_0}(\phi_0)-\lambda \phi_0) (x) \leq 0, x\in \sd/E_n$, for some $n$, then the martingale problem for $\cL^{\phi_0}$ is well-posed; cf. \cite[Theorem 6.7.1]{Pinsky}.
To find such a $\lambda$, \eqref{eq: Lphi0phi0} implies
\begin{align*}
 \cL^{\phi_0}\phi_0-\lambda \phi_0 &\leq  -\uc H_{8\ok \uc}(x) + \lambda \uc \log(\det(x)) - (\beta_1 \oc - 8\ok \alpha_1 \oc^2+ \lambda \oc) \norm{x} \, \indic_{\{\norm{x}>n_0+2\}} + C,\\
 & \leq \lambda \uc \log(\det(x)) - (\beta_1 \oc - 8\ok \alpha_1 \oc^2+ \lambda \oc) \norm{x} \, \indic_{\{\norm{x}>n_0+2\}} + C,
\end{align*}
where the second inequality follows from $H_{8\ok\uc} \geq H_\epsilon$, for
$8\ok\uc<\epsilon$, which is bounded from below on $\sd$ by Assumption
\ref{ass: H_eps} $i)$. For large enough $\lambda$, $\beta_1 \oc - 8\ok \alpha_1
\oc^2 + \lambda \oc >0$. Then, using \eqref{eq: lim_norm_det}, we conclude that $\cL^{\phi_0}\phi_0 \leq \lambda \phi_0$ outside a sufficiently large $E_n$.
\end{proof}

Let us now switch our attention to $\psi_0$ in Assumption \ref{ass: psi}.

\begin{lem}\label{lem: psi_0 sd}
For $\underline{k},\overline{k} > 0$ set
 \[
  \psi_0(x)\dfn \underline{k} \log(\det(x)) - \overline{k} \norm{x}\eta(\norm{x}), \quad x\in \sd.
 \]
Recall the constant $c_1$ from Assumption \ref{ass: H_eps}. Then, there exists
a $\overline{k}^h>0$ such that for all
$\overline{k} > \overline{k}^h$ and $\underline{k}>c_1^{-1}$, \eqref{eq: psi_0 hat v} is satisfied.
\end{lem}

\begin{proof}
$\lim_{x\rightarrow \partial \sd} \psi_0(x)
 =-\infty$ holds by \eqref{eq: lim_norm_det}. Since $\phi_0\geq 0$, this yields
 $\lim_{x\rightarrow \partial \sd} (\psi_0 - \phi_0)(x) =-\infty$. Hence it
 suffices to find $\underline{k}, \overline{k}>0$ such that
 $\lim_{x\rightarrow \partial \sd} \fF[\psi_0](x) = \infty$.

Set $\psi_0^{(1)}(x) \dfn \underline{k}\log(\det(x))$ and $\psi^{(2)}_0(x) \dfn
-\overline{k}\norm{x}\eta(\norm{x})$. By Assumption \ref{ass: coeff sd} $ii)$:
\begin{equation}\label{eq: fF big lb psi}
\begin{split}
\fF[\psi_0] \geq& \cL\psi^{(1)}_0 + \cL\psi^{(2)}_0 + V\\
&+
\frac{\uk}{2}\sum_{i,j,k,l=1}^{d}\pare{D_{(ij)} \psi^{(1)}_0+D_{(ij)}\psi^{(2)}_0}\tr(a^{ij}(a^{kl})')\pare{D_{(kl)}\psi^1_0+D_{(kl)}\psi^{(2)}_0}.
\end{split}
\end{equation}
From \eqref{eq: log_norm_ops}, for
$\norm{x}\geq n_0+2$,
 \begin{equation}\label{eq: psi10 big lb}
 \begin{split}
 \cL\psi^{(2)}_0 = -\frac{\overline{k}}{\norm{x}}\left(\tr(f'g)+\tr(f)\tr(g)
-\frac{2}{\norm{x}^2}\tr(fxgx) + \tr(B'x)\right).
\end{split}
\end{equation}
For the right-hand-side, $\tr(f'g) \leq \tr(f) \tr(g)$ for $f,g \in \sd$ and Assumption \ref{ass: growth sd} $i)$ imply that $\tr(f'g)+\tr(f)\tr(g)\leq 2\alpha_1\norm{x}$. Combining the previous inequality with $\tr(fxgx)>0$ and Assumption \ref{ass: growth sd} $ii)$, we obtain
\[
 \cL\psi^{(2)}_0 \geq C + \overline{k}\beta_1 \norm{x} \quad \text{ for } \norm{x}>n_0+2.
\]
On the other hand, when $\norm{x}\leq n_0+2$, similar to the discussion before \eqref{eq: Lphi2 ub}, one can show $\cL \psi^{(2)}_0 \geq C$. Therefore, the previous two estimates combined yield
\begin{equation}\label{eq: Lpsi2 lb}
 \cL\psi^{(2)}_0 \geq C + \overline{k} \beta_1 \norm{x} \, \indic_{\{\norm{x} >n_0+2\}} \quad \text{ for } x\in \sd.
\end{equation}
Bypassing $V$ for the moment, the quadratic term on the right hand side of
\eqref{eq: fF big lb psi} is estimated. We only consider $\{x: \norm{x}>n_0+2\}$ since the
quadratic term is nonnegative and we are looking for a lower bound. Here, $\psi^{(2)}_0(x) = -\overline{k}\norm{x}$ and hence \eqref{eq: log_norm_derivs} and \eqref{eq: quad_form_matrix} give
\begin{equation}\label{eq: quad psi lb}
\begin{split}
 &\sum_{i,j,k,l=1}^{d}\pare{D_{(ij)}
 \psi^{(1)}_0+D_{(ij)}\psi^{(2)}_0}\tr(a^{ij}(a^{kl})')\pare{D_{(kl)}\psi^{(1)}_0+D_{(kl)}\psi^{(2)}_0}\\
 &=4\underline{k}^2\tr(fx^{-1}gx^{-1}) - \frac{8\underline{k}\overline{k}}{\norm{x}}\tr(fx^{-1}gx) + \frac{4\overline{k}^2}{\norm{x}^2}\tr(fxgx)\\
 &\geq -8 \underline{k} \overline{k} \alpha_1 + 4 \frac{\overline{k}^2}{\norm{x}^2} \tr(fxgx),
\end{split}
\end{equation}
where the inequality holds due to $\tr(fx^{-1}gx^{-1})\geq 0$,  $\tr(fx^{-1}g
x)\leq \tr(f)\tr(x^{-1}gx)=\tr(f)\tr(g)$ (cf. the discussion after \eqref{eq: pos_def_trace}), and Assumption \ref{ass: growth sd} $i)$. Using $\cL
\psi^{(1)}_0 = \underline{k}H_0$ from \eqref{eq: log_norm_ops} and putting \eqref{eq: Lpsi2 lb}, \eqref{eq: quad psi lb} back to \eqref{eq: fF big lb psi} and utilizing Assumption \ref{ass: growth sd} $iii)$, we obtain
\[
 \fF[\psi_0] \geq \underline{k}H_0  + V\indic_{\norm{x}\leq n_0 + 2} + \bra{ 2
   \frac{\overline{k}^2\uk}{\norm{x}^2} \tr(fxgx) +(\overline{k}\beta_1 - \gamma_2) \norm{x}}\, \indic_{\{\norm{x} >n_0+2\}} + C.
\]

Consider when $\norm{x}$ is large. When $\beta_1>0$, $\tr(fxgx)>0$ and the
uniform lower bound for $H_0(x)$ on $\sd$ in Assumption \ref{ass: H_eps} $i)$ imply $\lim_{\norm{x}\uparrow\infty} \fF[\phi_0](x) =\infty$ for $\overline{k}> \gamma_2 /\beta_1$. On the other hand, when $\beta_1\leq 0$, Assumption \ref{ass: growth sd} $iv-a)$ gives $2\overline{k}^2 \uk \tr(fxgx)/(\norm{x}^2) + (\overline{k}\beta_1 -\gamma_2)\norm{x} \geq C+ (2\overline{k}^2 \uk \alpha_2 + \overline{k} \beta_1 - \gamma_2)\norm{x}$. Then taking $\overline{k}$ sufficiently large gives $\lim_{\norm{x}\uparrow\infty} \fF[\phi_0](x) =\infty$.

Consider now when $\norm{x}\leq n_0+2$ but $\det(x)\downarrow 0$. Note
$
\underline{k}H_0 + V = \left(H_0 + c_1V\right)/c_1 + (\underline{k}-c_1^{-1})H_0.
$
It then follows from  Assumption \ref{ass: H_eps} $i)$ and $iii)$ that $\lim_{\det(x)\downarrow 0} \underline{k}H_0(x) + V(x) = \infty$ when $\underline{k}>c_1^{-1}$, hence $\lim_{\det(x)\downarrow 0} \fF[\phi_0](x) =\infty$. Therefore, the first convergence in \eqref{eq: psi_0 hat v} is confirmed.
\end{proof}

Finally, it remains to verify \eqref{eq: main ub main ub}.

\begin{lem}\label{lem: M sd}
For the $\delta$ from Corollary \ref{cor: phi
  prop sd} $ii)$, there exists $\alpha>0$ such that \eqref{eq: main ub main ub}
holds.
\end{lem}

\begin{proof}
Using Lemma \ref{lem: psi_0 sd} and the construction of $\psi_0,\phi_0$, for any $K>0$
\begin{equation*}
\psi_0(x) + K \phi_0(x) = C-(K \uc - \underline{k})\log(\det(x)) +(K \oc- \overline{k})\norm{x}\eta(\norm{x}).
\end{equation*}
That the first inequality in \eqref{eq: main ub main ub} for large enough $K$ now follows from \eqref{eq: lim_norm_det}.
As for the second inequality in \eqref{eq: main ub main ub}, the same estimate as in \eqref{eq: new fF big ub} yields the existence of $\epsilon_0 > 0$ such that
\begin{equation*}
\fF[\delta\phi_0](x) \leq -\delta\underline{c} H_\epsilon(x) -\epsilon_0\norm{x}1_{\norm{x}>n_0 + 2} + C.
\end{equation*}
Then choose $\alpha>0$ such that $\alpha (\delta\overline{c}+\overline{k}) < \epsilon_0$ and $\alpha(1+\underline{k}/(\delta\underline{c})) < c_0$. It follows from the previous inequality and Lemma \ref{lem: psi_0 sd} that
\begin{equation*}
\begin{split}
\fF[\delta\phi_0] + \alpha(\delta\phi_0-\psi_0) &\leq -\delta\underline{c} H_\epsilon(x) -\alpha (\delta \underline{c} + \underline{k}) \log \det(x) - (\epsilon_0 -\alpha (\delta \overline{c} + \overline{k})) \norm{x} 1_{\norm{x}>n_0+2} +C\\
&\leq -\delta\underline{c}\left[H_{\epsilon}(x) +c_0
\log\det(x)\right] + C,
\end{split}
\end{equation*}
which is bounded from above when $\det(x)$ is small, due to Assumption \ref{ass: H_eps} $ii)$. If $\det(x)$ is bounded away from zero, both $H_\epsilon(x)$ and $\log\det(x)$ are bounded from below. Combining the previous two cases, we confirm the second inequality in  \eqref{eq: main ub main ub}.
\end{proof}

\nada{
With the previous preparation, we are ready to verify Assumption \ref{ass: h ub}.
\begin{cor}\label{cor: h ub sd}
There exists a constant $J$ such that $\mathcal{J}$ from \eqref{eq: J sd} satisfies \eqref{eq: h ub}.
\end{cor}

\begin{proof}

In light of Proposition \ref{prop: comparison}, it suffices to show that \eqref{eq: h ub} holds when $v_0  = \phi_0 + C$ for some
$C>0$.  In this case, it follows from \eqref{eq: pde h} and \eqref{eq: L_phi_def_T} that $h$ satisfies On the other hand, from the construction of $\phi_0$ and $\psi_0$ in Lemmas \ref{lem: quad bounds sd} and  \ref{lem: psi_0 sd}, there exists a sufficiently large $C$ such that
\begin{equation*}
h(0,x) = v_0(x) - \hat{v}(x) \leq \phi_0(x) - \psi_0(x) + C \leq C(1+\phi_0(x)).
\end{equation*}
The previous two inequalities combined with Lemma \ref{lem: main ub} yield
\begin{equation*}
\begin{split}
h(T,x) &\leq C\pare{1+\espalt{\prob^{v,x}_T}{}{\phi_0(X_T)}}\leq C\pare{1+K(1+\phi_0(x) +
\hat{v}_{-}(x))} \leq J(1+ \phi_0(x) + \hat{v}_-(x)),
\end{split}
\end{equation*}
for sufficiently large $J$. Finally, choosing $\phi$ in Proposition \ref{prop: e-existence} iii) as $\phi_0$, we have $e^{-\uk (\hat{v}- \phi_0)} \in \mathbb{L}^1(\sd, \hat{m})$, hence $e^{-\uk \hat{v}}\in \mathbb{L}^1(\sd, \hat{m})$ since $\phi_0$ is nonnegative. This implies $\hat{v}^- \in \mathbb{L}^1(\sd, \hat{m})$. On the other hand, $\phi_0\in \bL^1(\sd,\hat{m})$ (cf. Corollary \ref{cor: phi prop sd} $ii)$). Therefore $\mathcal{J}\in \mathbb{L}^1(\sd, \hat{m})$.
\end{proof}

}

\appendix
\section{Going between $\sd$ and $E$}\label{sec: identification}

This appendix shows how to consider \eqref{eq: pde sd} and \eqref{eq: e-eqn sd} as special
cases of \eqref{eq: pde E} and \eqref{eq: e-eqn E}, respectively.
Set $\td=d(d+1)/2$ and let $I : \{1, 2, \dots,
\td\}\mapsto \{(i,j)\,:\, i=1, \dots, d; j=i, \dots, d\}$ be a bijection such
that $I(p) = (p,p)$ for $p=1,...,d$.  If $I(p) = (i,j)$, we write $I'(p)=(j,i)$. Define $\ell: \sdall \rightarrow \Real^{\td}$ via $\ell(x)_p := x_{I(p)}$, for $p = 1,\dots,\td, x\in \sdall$. Thus, $\ell$ maps upper triangle entries of $x$ to entries in the vector $\ell(x)$. Denote by $\ell^{-1}$ the inverse of $\ell$.

Set $E= \ell(\sd)$. It can be shown that $E$ is an open, convex subset of $\Real^{\td}$ which can
be filled up by open, bounded sets $(E_n)_{n\in \Natural}$ with smooth
boundaries. Such $E_n$ is created by smoothing out the boundary of the set $\{y\in E\,:\, \det(\ell^{-1}(y)) >1/n, |y|<n\}$.

Given $X$ following \eqref{eq: sde sd}, one can then verify that $Y:=\ell(X)$
satisfies
\begin{equation*}
dY_t = \hat{B}(Y_t) \,dt + \hat{a}(Y_t) \,d\, \mbox{vec}(W_t),
\end{equation*}
where, for $y\in E$
\begin{align*}
 &\hat{B}_p(y) := B_{I(p)}(\ell^{-1}(y)),  \quad p = 1, \dots, \td,\\
 &\hat{a}_{pq}(y) := a^{I(p)}_{J(q)}(\ell^{-1}(y)),  \quad p=1, \dots, \td,
 q=1, \dots, d^2.
\end{align*}
Here, $J : \cbra{1,...,d^2}\mapsto \cbra{(i,j);i,j=1,...,d}$ is given by $J(1) = (1,1),\dots, J(d) = (d,1), J(d+1) = (1,2),\dots, J(2d) = (d,2), \dots, J(d,d) = d^2$.

Define $\hat{A}\dfn \hat{a}\hat{a}'$ and $\hat{\barA}_{pq}(y)\dfn \overline{A}_{I(p), I(q)}(\ell^{-1}(y))$ for $p,q=1, \cdots \td$ and $y\in \Real^{\td}$. Then Assumption \ref{ass: coeff sd} for
$A$ and $\bar{A}$ is equivalent to Assumption \ref{ass: coeff} for $\hat{A}$ and $\hat{\bar{A}}$.
Indeed, for any $\xi\in \Real^{\td}$, denote $\theta = \ell^{-1}(\xi)$. When $y = \ell(x)$,
\begin{equation*}
\begin{split}
 4 \sum_{p,q=1}^{\td} \xi_p (\hat{a} \hat{a}')_{pq}(y) \xi_q&= 4
 \sum_{p,q=1}^{\td} \xi_p \tr\pare{a^{I(p)}(a^{I(q)})'}(x) \xi_q\\
&=4\sum_{i,k=1}^{d}\theta_{ii}\tr(a^{ii}(a^{kk})')(x)\theta_{kk} +
4\sum_{i=1}^{d}\sum_{l=1}^{d}\sum_{k=1}^{l-1}\theta_{ii}\tr(a^{ii}(a^{kl})')(x)\theta_{kl}\\
&\qquad +
4\sum_{j=1}^{d}\sum_{i=1}^{j-1}\sum_{k=1}^{d}\theta_{ij}\tr(a^{ij}(a^{kk})')(x)\theta_{kk}
+
4\sum_{j=1}^{d}\sum_{i=1}^{j-1}\sum_{l=1}^{d}\sum_{k=1}^{l-1}\theta_{ij}\tr(a^{ij}(a^{kl})')(x)\theta_{kl}\\
&=\sum_{i,k=1}^{d}(2\theta_{ii})\tr(a^{ii}(a^{kk})')(x)(2\theta_{kk}) +
\sum_{i=1}^{d}\sum_{\stackrel{l,k=1}{l\neq k}}^{d}(2\theta_{ii})\tr(a^{ii}(a^{kl})')(x)\theta_{kl}\\
&\qquad +
\sum_{\stackrel{i,j=1}{i\neq j}}^{d}\sum_{k=1}^{d}\theta_{ij}\tr(a^{ij}(a^{kk})')(x)(2\theta_{kk})
+
\sum_{\stackrel{i,j=1}{i\neq j}, \stackrel{k,l=1}{k\neq l}}^{d}\theta_{ij}\tr(a^{ij}(a^{kl})')(x)\theta_{kl}\\
&= \sum_{i,j,k,l=1}^d(^{D}\theta)_{ij} \tr\pare{a^{ij} (a^{kl})'}(x)(^{D}\theta)_{kl},
\end{split}
\end{equation*}
where the third identity follows from $a^{ij}=a^{ji}$, and $^{D}\theta\in \sdall$ is obtained by doubling all diagonal entries of $\theta$. Note that $\norm{\theta}^2 \leq \norm{^{D}\theta}^2 \leq 2\norm{\theta}^2$.  Therefore Assumption \ref{ass: coeff sd} $i)$ for
$A$ is equivalent to Assumption \ref{ass: coeff} $i)$ for $\hat{A}$. The equivalence between Assumption \ref{ass: coeff sd} $ii)$ and Assumption \ref{ass: coeff} $ii)$ can be proved similarly.

Now let us connect operators $\fF$ in \eqref{eq: F E} and \eqref{eq: F sd}.
Let $g$ be a smooth function on $\sd$ and define $\tilde{g}:E\rightarrow \Real$ by
$\tilde{g}(y) := g(x)$ where $x=\ell^{-1}(y)$.  Calculations show that $\partial_p\tilde{g}(y)= D_{I(p)}g(\ell^{-1}(y))$ when $I(p)$ is diagonal, or $(D_{I(p)} +  D_{I'(p)})g(\ell^{-1}(y))$ when $I(p)$ is off-diagonal.
It then follows that
\begin{equation*}
\begin{split}
\sum_{p=1}^{\td}\hat{B}_p(y)\partial_p\tilde{g}(y) &=\sum_{i=1}^{d}B_{ii}(x)D_{(ii)}g(x) +
\sum_{j=1}^{d}\sum_{i=1}^{j-1}B_{ij}(x)D_{(ij)}g(x) + \sum_{j=1}^{d}\sum_{i=1}^{j-1} B_{ij}(x)D_{(ji)}g(x)\\
&=\sum_{i,j=1}^{d}B_{ij}(x)D_{(ij)}g(x),
\end{split}
\end{equation*}
where the second identity above follows from $B_{ij} = B_{ji}$.  A similar (but longer) calculation using $a^{ij} = a^{ji}$ and
$\bar{A}_{(ij),(kl)} = \bar{A}_{(ji),(kl)} = \bar{A}_{(ij),(lk)}$ (cf. \eqref{eq: barA cond}) shows
\begin{equation*}
\begin{split}
\sum_{p,q=1}^{\td}
\left(\hat{a}\hat{a}'\right)_{pq}(y)\partial^2_{pq}\tilde{g}(y)& =
\sum_{i,j,k,l=1}^{d}\tr(a^{ij}(a^{kl})')(x)D^2_{(ij),(kl)}g(x),\\
\sum_{p,q=1}^{\td}\partial_p\tilde{g}(y)\hat{\bar{A}}_{pq}(y)\partial_{q}\tilde{g}(y)
&= \sum_{i,j,k,l=1}^{d}D_{(ij)}g(x)\bar{A}_{(ij),(kl)}(x)D_{(kl)}g(x).
\end{split}
\end{equation*}
Write $\hat{V}(y) = V(x)$ where $x = \ell^{-1}(y)$. The previous three identities combined yield $\fF[g](x) =
\fF[\tilde{g}](\ell(x))$.  Therefore, \eqref{eq: pde sd} and \eqref{eq: e-eqn sd} can be considered as special cases of \eqref{eq: pde E} and \eqref{eq: e-eqn E}, respectively.

\bibliographystyle{siam}
\bibliography{biblio}
\end{document}